\documentclass[12pt]{amsart}
\usepackage{amssymb, amsmath, amsfonts, amsthm}
\usepackage{a4wide}
\usepackage{mathrsfs} 
\usepackage{mathtools}
\usepackage{enumerate}
\usepackage{physics}

\newtheorem{theorem}{Theorem}[section]
\newtheorem*{theorem*}{Theorem}

\newtheorem{corollary}[theorem]{Corollary}
\newtheorem{lemma}[theorem]{Lemma}
\newtheorem{rem}[theorem]{Remark}

\newtheorem{proposition}[theorem]{Proposition}

\theoremstyle{definition}
\newtheorem{definition}[theorem]{Definition}

\usepackage{color}

\newcommand{\rr}{\mathbb{R}}
\newcommand{\nn}{\mathbb{N}}

\newcommand{\ee}{\varepsilon}

\newcommand{\C}{\mathbb{C}}

\newcommand{\eee}{\mathcal{E}}
\newcommand{\EEE}{\mathbb{E}}
\newcommand{\rme}{\mathrm{e}}

\newcommand{\K}{\mathbb{K}}

\newcommand{\N}{\mathbb{N}}

\newcommand{\R}{\mathbb{R}}

\newcommand{\n}[1]{\| #1\|}

\newcommand{\e}{\varepsilon}

\renewcommand{\frown}{\mathbin{\raisebox{0.3ex}{$\smallfrown$}}}

\renewcommand{\leq}{\leqslant}
\renewcommand{\geq}{\geqslant}

\begin{document}
\title[Asymptotic ideal seminorms]{Operator ideals and three-space properties of asymptotic ideal seminorms}
\author[R.M. Causey]{Ryan M. Causey}
\address{Department of Mathematics, Miami University, Oxford, OH 45056, USA}
\email{causeyrm@miamioh.edu}

\author[S. Draga]{Szymon Draga}
\address{Institute of Mathematics, Czech Academy of Sciences, \v{Z}itn\'a 25, 115 67 Prague 1, Czech Republic}
\email{szymon.draga@gmail.com}

\author[T. Kochanek]{Tomasz Kochanek}
\address{Institute of Mathematics, Polish Academy of Sciences, \'Sniadeckich 8, 00-656 Warsaw, Poland\, {\rm and}\, Institute of Mathematics, University of Warsaw, Banacha~2, 02-097 Warsaw, Poland}
\email{tkoch@impan.pl}

\thanks{Research of the second-named author was supported by the GA\v{C}R project 16-34860L and RVO: 67985840.}


\begin{abstract}
We introduce asymptotic analogues of the Rademacher and martingale type and cotype of Banach spaces and operators acting on them. Some classical local theory results related, for example, to the~`automatic-type' phenomenon, the type-cotype duality, or the Maurey--Pisier theorem, are extended to the asymptotic setting. We also investigate operator ideals corresponding to the asymptotic subtype/subcotype. As an~application of this theory, we provide a~sharp version of a~result of Brooker and Lancien by showing that any twisted sum of Banach spaces with Szlenk power types $p$ and $q$ has Szlenk power type $\max\{p,q\}$.
\end{abstract}

\maketitle

\tableofcontents
\addtocontents{toc}{\protect\setcounter{tocdepth}{1}}

\section{Introduction}
Beginning with the result of James \cite{James} that every uniformly convex Banach space is superreflexive, and the converse of Enflo \cite{Enflo} that every superreflexive Banach space is isomorphic to a uniformly convex Banach space, a beautiful theory has emerged regarding the renorming of superreflexive Banach spaces.  In particular, Pisier \cite{pisier} proved that every superreflexive Banach space can be renormed to be uniformly convex with some power type modulus, and demonstrated a~method by which one can determine the supremum of the corresponding power types by means of Haar cotype quantities. Since then, Haar, martingale, and Rademacher type and cotype have played important roles in the local theory of Banach spaces. One remarkable result regarding these topics is the solution of several three-space problems by Enflo, Lindenstrauss, and Pisier \cite{ELP}. Namely, they showed that if $p_X$ denotes the supremum of those $p$ for which $X$ has Haar type $p$ (resp. Rademacher type $p$), then for any Banach space $X$ and any closed subspace $Y$ of $X$, we have $p_X=\max\{p_Y, p_{X/Y}\}$. Further elegant results regarding Haar, martingale, Rademacher and Gaussian type and cotype of operators were obtained by Beauzamy \cite{Beau}, Hinrichs \cite{Hinrichs} and Wenzel \cite{Wenzel}. 

Each of the topics from the previous paragraph is of a~local nature. The goal of this work is to prove the asymptotic analogues of these local results. We introduce asymptotic analogues of both the martingale and Rademacher type/cotype ideal norms and study duality, renorming, ideal properties, and three-space properties associated with these ideal norms. An important asymptotic renorming result is that a~Banach space is asymptotically uniformly smoothable if and only if its Szlenk index does not exceed $\omega$, the first infinite ordinal, and the modulus of asymptotic uniform smoothness can be taken to have power type (this is due to Knaust, Odell, and Schlumprecht \cite{KOS} for separable spaces and Raja \cite{R} for the general case). Furthermore, Godefroy, Kalton and Lancien \cite{GKL} used an~isomorphic invariant, the Szlenk power type $\mathsf{p}(X)$ of $X$, to determine for which $p$ a~given separable, asymptotically uniformly smoothable Banach space $X$ can be renormed to have modulus of power type $p$. Their result was extended to nonseparable spaces, higher ordinals, and operators in \cite{C}. Moreover, Brooker and Lancien \cite{BL} showed that being asymptotically uniformly smoothable is a~three-space property and they gave the quantitative estimate $\mathsf{p}(X)\leqslant \mathsf{p}(Y)+\mathsf{p}(X/Y)$ for any Banach space $X$ and any closed subspace $Y$ of $X$ such that both $Y$ and $X/Y$ are asymptotically uniformly smoothable. 
We first define the notions of block basic  and block type/cotype for a collection of formal identities between basic sequences, meant to parallel the notions of martingale and Rademacher type. We prove block analogues of several well-known local results regarding martingale type and cotype. We then define the $n^{\mathrm{th}}$ asymptotic structure of an~operator and use the notions of asymptotic basic type/cotype and asymptotic type/cotype, as well as weak$^\ast$ analogues of these notions for adjoints. Each of these notions has a~parameter $p$, which ranges over the interval $[1, \infty]$. Regarding these results, we obtain the following (for details see Section~4 and Theorem~\ref{ideal_theorem}). 
\begin{theorem} For any $1\leqslant p,q\leqslant \infty$, there is an ideal norm $\mathfrak{a}_p$ on the class of operators $\mathfrak{A}_p$ with asymptotic basic type $p$ making $(\mathfrak{A}_p, \mathfrak{a}_p)$ into a Banach ideal. Analogous results hold for the classes of operators with asymptotic basic cotype $q$, asymptotic type $p$, asymptotic cotype $q$, or any of the analogous weak$^\ast$ classes.
\end{theorem}

In perfect parallel with the local case, any asymptotically uniformly smoothable Banach space automatically has some nontrivial asymptotic basic type $p$, but there are asymptotically uniformly smoothable operators which have no nontrivial asymptotic basic type ({\it cf.} Theorem~\ref{causey_glk}, Remarks~\ref{op_nontrivial} and \ref{Mr.K}). This fact is analogous to Pisier's result that any superreflexive Banach space has nontrivial martingale type $p$ for some $1<p<2$, while it is easy to construct super weakly compact operators which do not have martingale type $p$ for any $p>1$ (see \cite{Wenzel}).

Analogous assertions for asymptotic basic cotype and asymptotic type/cotype can be formulated in terms of the notions of asymptotic (basic) subtype/subcotype which we define following the results of Beauzamy, Hinrichs and Wenzel. Regarding these notions, we prove the following theorem, where clause (i) follows from Theorem~\ref{main1}(i) and Remark~\ref{Szlenk remark}, whereas clauses (ii) and (iii) follow from Theorem~\ref{main1}(iii) and (iv), respectively. (Here $Sz(A)$ denotes the Szlenk index of an~operator $A$.)
\begin{theorem}\label{subtype_char}
Let $A\colon X\to Y$ be an operator between Banach spaces $X$ and $Y$. 
\begin{enumerate}[{\rm (i)}]
\item $A$ has asymptotic basic subtype if and only if $Sz(A)\leqslant \omega$.
\item Either $A$ has asymptotic subtype or the identity on $\ell_1$ is crudely asymptotically finitely representable in $A$, and these two alternatives are exclusive. Moreover, the class of operators with asymptotic subtype is a~closed operator ideal. 
\item Either $A$ has asymptotic subcotype or the identity on $c_0$ is crudely asymptotically finitely representable in $A$, and these two alternatives are exclusive. Moreover, the class of operators with asymptotic subcotype is a~closed operator ideal.
\end{enumerate}
\label{subtype}
\end{theorem}

Restricting our attention to Banach spaces (that is, taking $A$ to be the identity operator), we prove the following `automatic power type' phenomenon which in general fails for operators. It follows from Corollary~\ref{aut_type} applied to appropriate identity blocks ({\it cf.} also Theorem~\ref{main2}).
\begin{theorem} A Banach space $X$ has asymptotic basic subtype if and only if there exists $1<p<\infty$ such that $X$ has asymptotic basic type $p$. The analogous result holds for asymptotic basic subcotype, asymptotic subtype, asymptotic subcotype, and the corresponding weak$^\ast$ properties. 
\label{automatic power type}
\end{theorem}

We also obtain the following result regarding duality (see Corollary~\ref{duality_for_operators}).
\begin{theorem} 
An operator $A\colon X\to Y$ between Banach spaces $X$ and $Y$ has asymptotic basic subtype {\rm (}resp. asymptotic basic type $p${\rm )} if and only if $A^\ast$ has weak$^\ast$ asymptotic basic subcotype {\rm (}resp. weak$^\ast$ asymptotic basic cotype $q$, where $1/p+1/q=1${\rm )}. 
\end{theorem}

Similarly to the result mentioned earlier, there exist operators with asymptotic basic subcotype or asymptotic subtype/subcotype and without nontrivial corresponding power type (see
Remark~\ref{op_nontrivial}). Note also that Theorem~\ref{subtype_char} yields a~structural characterization of operators with asymptotic subtype/subcotype and that the~corresponding result in the classical Rademacher type/cotype case would be that a~Banach space either has some nontrivial type (resp. cotype) or contains $\ell_1^n$'s (resp. $\ell_\infty^n$'s) uniformly, while it is easy to construct operators with no nontrivial type (resp. cotype) which do not uniformly preserve $\ell_1^n$'s. 

The paper is concluded with a quantitative strengthening of the result of Brooker and Lancien mentioned earlier, as well as a~proof of the analogous fact for asymptotic type. In what follows, let $\mathsf{q}(X)$ denote the supremum of those $p$ such that $X$ has asymptotic basic type $p$ and let $\mathsf{t}(X)$ denote the supremum of those $p$ such that $X$ has asymptotic type $p$. Combining Theorems~\ref{subtype} and \ref{automatic power type} with an~asymptotic version of the well-known blocking arguments of James, we obtain that $Sz(X)\leqslant \omega$ if and only if $\mathsf{q}(X)>1$ and $\ell_1$ fails to be asymptotically finitely representable in $X$ if and only if $\mathsf{t}(X)>1$. The aforementioned result is contained in Theorem~\ref{BL_main}, where the assertion for $\mathsf{q}(X)$ is presented in a~dual form ({\it cf.} Remark~\ref{Mr.K}).
\begin{theorem} 
Let $X$ be a Banach space and $Y$ a closed subspace of $X$.
Then
$$\mathsf{q}(X)=\min\{\mathsf{q}(Y), \mathsf{q}(X/Y)\}\quad\mbox{ and }\quad \mathsf{t}(X)=\min\{\mathsf{t}(Y), \textbf{\emph{t}}(X/Y)\}.
$$
\end{theorem}

\section{Block structures} 
In what follows, $\K$ is either $\R$ or $\C$ and $(\rme_i)_{i=1}^n$ stands for the canonical basis of $\K^n$. For any Banach space $X$, the symbols $B_X$ and $S_X$ stand for its unit ball and unit sphere, respectively. All operators are assumed to be linear and bounded. For any sequence $(x_i)$ in a~Banach space,  we denote by $[x_i]$ its closed linear span.

Given $a\geqslant 0$, $b \geqslant 1$ and $n\in\N$, we denote by $\mathcal{B}^{a,b}_n$ the set of all pairs $(r, \rho)$ of seminorms on $\mathbb{K}^n$ such that: 
\begin{enumerate}[(i)]
\item $r$ is a norm;
\item $r(\rme_i)=1$ for each $1\leqslant i\leqslant n$;
\item for all $1\leqslant m\leqslant n$ and all scalars $(a_i)_{i=1}^n$ we have
$$
r\Bigl(\sum_{i=1}^m a_i \rme_i\Bigr) \leqslant br\Bigl(\sum_{i=1}^n a_i \rme_i\Bigr)\quad\mbox{ and }\quad\rho\Bigl(\sum_{i=1}^m a_i \rme_i\Bigr)\leqslant b\rho\Bigl(\sum_{i=1}^n a_i \rme_i\Bigr);
$$
\item for any scalars $(a_i)_{i=1}^n$ we have 
$$
\rho\Bigl(\sum_{i=1}^n a_i \rme_i\Bigr)\leqslant a r\Bigl(\sum_{i=1}^n a_i \rme_i\Bigr).
$$
\end{enumerate}
We endow $\mathcal{B}^{a,b}_n$ with the metric 
\begin{equation*}
\begin{split}
d_n((r, \rho),(s, \sigma))=\max_{(a_i)_{i=1}^n\in B_{\ell_\infty^n}}\!\! \max\Biggl\{\Bigl|r\Bigl(\sum_{i=1}^n a_i \rme_i\Bigr) &-s\Bigl(\sum_{i=1}^n a_i \rme_i\Bigr)\Bigr|,\\
& \Bigl|\rho\Bigl(\sum_{i=1}^n a_i \rme_i\Bigr)-\sigma\Bigl(\sum_{i=1}^n a_i \rme_i\Bigr)\Bigr|\Biggr\}.
\end{split}
\end{equation*}

\begin{definition}\label{block_def}
Given $(r, \rho)\in \mathcal{B}^{a,b}_n$ and $(s,\sigma)\in \mathcal{B}^{a,b}_m$, we say that  $(s, \sigma)$ is a~\emph{block} of $(r, \rho)$ provided there exist a~scalar sequence $(a_i)_{i=1}^n$ and integers $0=k_0<\ldots <k_m\leqslant n$ such that $r(\sum_{k_{j-1}<i\leq k_j} a_i \rme_i)=1$ for each $1\leqslant j\leqslant m$ and 
$$
s\Bigl(\sum_{j=1}^m b_j \rme_j\Bigr)= r\Bigl(\sum_{j=1}^m \sum_{i=k_{j-1}+1}^{k_j} b_j a_i\rme_i\Bigr),\qquad \sigma\Bigl(\sum_{j=1}^m b_j \rme_j\Bigr)= \rho\Bigl(\sum_{j=1}^m \sum_{i=k_{j-1}+1}^{k_j} b_j a_i\rme_i\Bigr)
$$  
for every scalar sequence $(b_j)_{j=1}^m$.
\end{definition}

Let $X$, $Y$ be Banach spaces, $(x_i)_{i=1}^\infty$ a~normalized basic sequence in $X$ and $A\colon X\to Y$ an~operator such that $(Ax_i\colon i\in\N, Ax_i\neq 0)$ is also basic, where the order corresponds to the order of $(x_i)_{i=1}^\infty$. Then, there is $b\geq 1$ such that for all $1\leq m\leq n$ and every scalar sequence $(a_i)_{i=1}^n$ we have 
$$
\Biggl\|\sum_{i=1}^m a_i x_i\Biggr\|\leqslant b\Biggl\|\sum_{i=1}^n a_ix_i\Biggr\|\quad\mbox{ and }\quad \Biggl\|A\sum_{i=1}^m a_i x_i \Biggr\|\leqslant b\Biggl\|A\sum_{i=1}^n a_i x_i\Biggr\|.
$$ 
Hence, for each $n\in\N$, we can associate with the operator $A\vert_{[x_i\colon i\leq n]}$ a~member of $\mathcal{B}^{\|A\|, b}_n$ defined by 
$$
r\Bigl(\sum_{i=1}^n a_i \rme_i\Big)=\Biggl\|\sum_{i=1}^n a_i x_i\Biggr\|\quad\mbox{ and }\quad\rho\Big(\sum_{i=1}^n a_i \rme_i\Big)= \Biggl\|A\sum_{i=1}^n a_i x_i\Biggr\|.
$$
Let $\mathcal{E}\subset\bigcup_{n=1}^\infty\mathcal{B}_n^{a,b}$. We adopt the convention of saying that $A\colon [x_i\colon 1\leqslant i\leqslant n]\to  [Ax_i\colon 1\leqslant i\leqslant n]$ belongs to $\mathcal{E}$, whenever the pair $(r, \rho)$ defined above belongs to $\mathcal{E}$. If it holds true with $A$ being the identity operator, then we say that $(x_i)_{i=1}^n\in \mathcal{E}$.

Recall that if $(x_i)_{i=1}^\infty$ is a~basic sequence, then by a~{\it block sequence} with respect to $(x_i)_{i=1}^\infty$ we mean any sequence $(z_i)_{i=1}^\infty$ of nonzero vectors for which there exist integers $0=p_0<p_1<\ldots$ and a~scalar sequence $(a_i)_{i=1}^\infty$ such that 
$$
z_i=\sum_{j=p_{i-1}+1}^{p_i} a_jx_j\quad\mbox{for }i\in\N.
$$

\begin{definition}
Let $(x_i)_{i=1}^\infty$ be a normalized basic sequence in a~Banach space and let $A$ be an~operator defined on $[x_i\colon i\in\N]$ such that $(Ax_i:i\in \nn, Ax_i\neq 0)$ is also basic. Let also $a\geq 0$, $b\geq 1$ and $\mathcal{E}$ be any collection contained in $\bigcup_{n=1}^\infty\mathcal{B}_n^{a,b}$. We call the~operator $A$:
\begin{itemize}
\item[(i)] {\it block finitely representable on} $\mathcal{E}$ if for each $n\in\N$ the~restriction $A\colon [x_i\colon 1\leqslant i\leqslant n]\to [Ax_i\colon 1\leqslant i\leqslant n]\in \mathcal{E}$;
\item[(ii)] {\it shrinking} if for any bounded block sequence $(z_i)_{i=1}^\infty$ with respect to $(x_i)_{i=1}^\infty$ the~sequence $(Az_i)_{i=1}^\infty$ is weakly null;
\item[(iii)] {\it boundedly complete} if for any block sequence $(z_i)_{i=1}^\infty$ with respect to $(x_i)_{i=1}^\infty$ such that $\inf_n \|Az_n\|>0$ we have $\sup_n \|\sum_{i=1}^n z_i\|=\infty$. 
\end{itemize}
We say that the sequence $(x_i)_{i=1}^\infty$ has any of the~properties listed above if the~identity operator on $[x_i\colon i\in\N]$ has the~corresponding property.
\end{definition}

\begin{definition}\label{def_collection}
Let $a\geq 0$, $b\geq 1$. We say a collection $\mathcal{E}\subset \bigcup_{n=1}^\infty \mathcal{B}^{a,b}_n$ is:
\begin{enumerate}[(i)]
\item \emph{closed} if for each $n\in \nn$, $\mathcal{E}\cap \mathcal{B}^{a,b}_n$ is closed in $\mathcal{B}^{a,b}_n$ with respect to the metric $d_n$;
\item \emph{block closed} if $\mathcal{E}$ contains all blocks of its members;
\item \emph{stable} if it is both closed and block closed;
\item \emph{nontrivial} if for each $n\in \nn$ we have $\mathcal{E}\cap \mathcal{B}^{a,b}_n\neq \varnothing$;
\item an \emph{identity block} if it is stable and $r=\rho$ for every $(r, \rho)\in \mathcal{E}$.
\end{enumerate}
\end{definition}

Each member of an identity block $\eee$ has the form $(r,r)$ with some norm $r$ on $\mathbb{K}^n$. Hence, it can be regarded as the normed space $(\mathbb{K}^n, r)$ for some $n\in \nn$.  Hereinafter, we will use the~convention of referring to elements of $\eee$ as normalized basic sequences in $\mathbb{K}^n$. We do not necessarily need to mention the underlying norm explicitly; sometimes we will denote it by the~usual symbol $\n{\!\cdot\!}$ rather than $r$. Blocks in the~sense of Definition~\ref{block_def} will be regarded as block basic sequences.

Now, we will define two `flavors' of seminorms, each of which having a~parameter $p\in [1,\infty]$, each having both a~type and a~cotype version. Later  we will use these quantities to define both weak and weak$^*$ asymptotic ideal seminorms. One flavor leads to an~asymptotic analogue of the local notions of Rademacher type/cotype, the results of which can be compared to those found in \cite{Beau, Hinrichs}, while the other leads to an~asymptotic analogue of the notion of martingale type, the results of which can be compared to those found in \cite{Wenzel}.

In the sequel, $(\e_i)_{i=1}^\infty$ stands for a~sequence of independent and identically distributed Rademacher random variables. Specifically, we can choose $\e_i$ as~the $i$th coordinate function on the~probability space $\{\pm 1\}^\N$.

\begin{definition}
Let $a\geq 0$, $b\geq 1$ and $\mathcal{E}\subset\bigcup_{k=1}^\infty\mathcal{B}_k^{a,b}$ be a~stable set. For $1\leq p\leq \infty$ and $n\in\N$ we define $\alpha_{p,n}(\mathcal{E})$ (respectively: $\tau_{p,n}(\mathcal{E})$) as the~infimum over all constants $T>0$ such that for every $(r, \rho)\in \mathcal{E}\cap \mathcal{B}^{a,b}_n$ and all scalars $(a_i)_{i=1}^n$ we have
$$
\rho\Big(\sum_{i=1}^n a_i \rme_i\Big)\leqslant T \|(a_i)_{i=1}^n\|_{\ell_p^n};
$$  
$$
\mbox{resp. }\,\Biggl\|\rho\Big(\sum_{i=1}^n \ee_i a_i \rme_i\Big)\Biggr\|_{L_p}\leqslant T \|(a_i)_{i=1}^n\|_{\ell_p^n}.
$$
Similarly, for $1\leq q\leq\infty$ and $n\in\N$ we define $\beta_{q,n}(\mathcal{E})$ (respectively: $\gamma_{q,n}(\mathcal{E})$) as the~infimum over all constants $C>0$ such that for every $(r, \rho)\in \mathcal{E}\cap \mathcal{B}^{a,b}_n$ and all scalars $(a_i)_{i=1}^n$ we have
$$
C r\Bigl(\sum_{i=1}^n a_i \rme_i\Bigr) \geqslant \|(a_i \rho(\rme_i))_{i=1}^n\|_{\ell_q^n}; 
$$
$$
\mbox{resp. }\,C \Biggl\|r\Big(\sum_{i=1}^n \ee_i a_i \rme_i\Big)\Biggr\|_{L_q} \geqslant \|(a_i \rho(\rme_i))_{i=1}^n\|_{\ell_q^n}.
$$
\end{definition}

Now, we introduce some further quantities which are used to test whether a~given collection $\mathcal{E}$ has block (basic) subtype/subcotype without specifying a~concrete exponent.
\begin{definition}
Let $a\geq 0$, $b\geq 1$ and $\mathcal{E}\subset\bigcup_{k=1}^\infty\mathcal{B}_k^{a,b}$ be a~stable set. For $n\in\N$ we define $\theta_n(\mathcal{E})$ (respectively: $\Theta_n(\mathcal{E})$) as the infimum over all constants $T>0$ such that for every $(r, \rho)\in \mathcal{E}\cap \mathcal{B}^{a,b}_n$ and all scalars $(a_i)_{i=1}^n$ we have
$$
\rho\Big(\frac{1}{n}\sum_{i=1}^n a_i \rme_i\Big)\leqslant T \|(a_i)_{i=1}^n\|_{\ell_\infty^n};
$$  
$$
\mbox{resp. }\,\Biggl\|\rho\Big(\frac{1}{n}\sum_{i=1}^n \ee_i a_i \rme_i\Big)\Biggr\|_{L_1}\leqslant T \|(a_i)_{i=1}^n\|_{\ell_\infty^n}.
$$
Similarly, we define $\upsilon_n(\eee)$ (respectively: $\Upsilon_n(\eee)$) as the~infimum over all constants $C>0$ such that for every $(r, \rho)\in \mathcal{E}\cap \mathcal{B}^{a,b}_n$ and all scalars $(a_i)_{i=1}^n$ we have
$$
C r\Bigl(\sum_{i=1}^n a_i \rme_i\Bigr) \geqslant \|(a_i \rho(\rme_i))_{i=1}^n\|_{\ell_{-\infty}^n}; 
$$
$$
\mbox{resp. }\,C \Biggl\|r\Big(\sum_{i=1}^n \ee_i a_i \rme_i\Big)\Biggr\|_{L_1} \geqslant \|(a_i \rho(\rme_i))_{i=1}^n\|_{\ell_{-\infty}^n},
$$
where $\|(b_i)_{i=1}^n\|_{\ell_{-\infty}^n}:=\min_{1\leq i\leq n}\abs{b_i}$.
\end{definition}

For any given $1\leq p\leq\infty$ we denote by $p^\prime$ its conjugate exponent, so that $1/p+1/p^\prime=1$ with the convention that $p^\prime=\infty$ for $p=1$ and $p^\prime=1$ for $p=\infty$.

\begin{definition}
Let $a\geq 0$, $b\geq 1$. For any $1\leq p,q\leq \infty$ we say that a~stable set $\mathcal{E}\subset\bigcup_{k=1}^\infty\mathcal{B}_k^{a,b}$ has:
\begin{enumerate}[(i)]
\item \emph{block basic type} $p$ if $\sup_n \alpha_{p,n}(\mathcal{E})<\infty$, 

\item \emph{block type} $p$ if $\sup_n \tau_{p,n}(\mathcal{E})<\infty$, 

\item \emph{block basic subtype} $p$ if $\alpha_{p,n}(\eee)=\mathrm{o}(n^{1/p^\prime})$, {\it i.e.} $\lim_n n^{1/p-1}\alpha_{p,n}(\mathcal{E})=0$, 

\item \emph{block subtype} $p$ if $\tau_{p,n}(\eee)=\mathrm{o}(n^{1/p^\prime})$, {\it i.e.} $\lim_n n^{1/p-1}\tau_{p,n}(\mathcal{E})=0$, 

\item \emph{block basic subtype} if $\lim_n \theta_n(\mathcal{E})=0$, 

\item \emph{block subtype} if $\lim_n\Theta_n(\mathcal{E})=0$,

\item \emph{block basic cotype} $q$ if $\sup_n \beta_{q,n}(\mathcal{E})<\infty$,

\item \emph{block cotype} $q$ if $\sup_n \gamma_{q,n}(\mathcal{E})<\infty$,

\item \emph{block basic subcotype} $q$ if $\beta_{q,n}(\eee)=\mathrm{o}(n^{1/q})$, 

\item \emph{block subcotype} $q$ if $\gamma_{q,n}(\eee)=\mathrm{o}(n^{1/q})$, 

\item \emph{block basic subcotype} if $\lim_n \upsilon_n(\mathcal{E})=0$,

\item \emph{block subcotype} if $\lim_n \Upsilon_n(\mathcal{E})=0$.
\end{enumerate}
\end{definition}

The terminology introduced in (v), (vi), (xi) and (xii) is justified in the sense that having block (basic) subtype/subcotype is equivalent to having block (basic) subtype/subcotype for some exponent from $(1,\infty)$ (a~{\it nontrivial} one). We will prove these facts in Theorems~\ref{bbtype_thm}, \ref{btype_thm}, \ref{bbco_thm} and \ref{bco_thm} below.

We note that, unlike the case of the martingale and Rademacher type, every $p\in [1,\infty]$ can be realized as the block type or block basic type of some collection. Also, unlike martingale and Rademacher cotype, every $q\in [1,\infty]$ can be the block cotype or block basic cotype of some collection. Indeed, if for each $k\in\N$ we let $\mathcal{E}\cap \mathcal{B}_k^{1,1}$ be the canonical $\ell_q^k$ basis, then $\mathcal{E}$ has block type $q$, block basic type $q$, block cotype $q$, and block basic cotype $q$. 

The notion of Rademacher subtype and subcotype as defined here are block versions of definitions due to Beauzamy \cite{Beau}. Subsequently, the notions of  (Rademacher, Gaussian, martingale, and Haar) subcotype were studied extensively by Hinrichs \cite{Hinrichs} and Wenzel \cite{Wenzel}.  

Recall that a~sequence $(a_n)_{n=1}^\infty$ of real numbers is called {\it submultiplicative} if $a_{mn}\leq a_m a_n$ for all $m,n\in\N$. The~following fact is a a~well-known phenomenon and can be proved by methods similar to those used for classical Rademacher type/cotype constants (see, {\it e.g.}, \cite[Prop.~G.3]{benyamini}). We show it only for the~block basic type constants (this case will be used in the~proof of our main result, Theorem~\ref{BL_main}); in other cases the proof is similar.
\begin{proposition}\label{submP}
Let $a\geqslant 0$, $b\geqslant 1$ and let $\mathcal{E}\subset \cup_{n=1}^\infty \mathcal{B}^{a,b}_n$ be an~identity block. For every $1\leq p\leq\infty$, all the sequences: $(\alpha_{p,n}(\eee))_{n=1}^\infty$, $(\tau_{p,n}(\eee))_{n=1}^\infty$, $(\beta_{p,n}(\eee))_{n=1}^\infty$, $(\gamma_{p,n}(\eee))_{n=1}^\infty$, $(\theta_{p,n}(\eee))_{n=1}^\infty$, $(\Theta_{p,n}(\eee))_{n=1}^\infty$, $(\upsilon_{p,n}(\eee))_{n=1}^\infty$, $(\Upsilon_{p,n}(\eee))_{n=1}^\infty$ are submultiplicative.
\end{proposition}
\begin{proof}
Fix $m,n\in\N$, $1\leq p\leq\infty$ and a~scalar sequence $(a_i)_{i=1}^{mn}$. For $1\leq i\leq n$ define $z_i=\sum_{j=(i-1)m+1}^{im} a_j\rme_j$ and assume for now that every $z_i\neq 0$. Then, for any $(r,r)\in \mathcal{E}\cap\mathcal{B}^{a,b}_{mn}$, the~vectors $z_i/r(z_i)$ produce a~block $(s,s)\in \mathcal{E}\cap\mathcal{B}^{a,b}_n$ and hence
\begin{equation*}
    \begin{split}
        r\Big(\sum_{i=1}^{mn}a_i\rme_i\Big) &=r\Big(\sum_{i=1}^n r(z_i)\frac{z_i}{r(z_i)}\Big)\leq \alpha_{p,n}(\eee)\n{(r(z_i))_{i=1}^n}_{\ell_p^n}\\
        & \leq \alpha_{p,n}(\eee)\alpha_{p,m}(\eee)\n{(\n{(a_j)_{j=(i-1)m+1}^{im}}_{\ell_p^m})_{i=1}^n}_{\ell_p^n}\\   &=\alpha_{p,n}(\eee)\alpha_{p,m}(\eee)\n{(a_i)_{i=1}^{mn}}_{\ell_p^{mn}}.
    \end{split}
\end{equation*}
In the general case, we repeat the same estimation omitting all $i$'s for which $z_i=0$, and we obtain a~right-hand side as above but with $\alpha_{p,k}(\eee)$ instead of $\alpha_{p,n}(\eee)$, where $k<n$. However, since $\eee$ is an~identity block, we have $\alpha_{p,k}(\eee)\leq \alpha_{p,n}(\eee)$. Hence, in both cases the above estimate is valid, which gives $\alpha_{p,mn}(\eee)\leq \alpha_{p,m}(\eee)\alpha_{p,n}(\eee)$.
\end{proof}

We will also need the following, folklore lemma.
\begin{lemma}\label{folklore}
If $(a_n)_{n=1}^\infty$ is a~monotone increasing, submultiplicative sequence of nonnegative numbers and for some $k>1$ and $1\leq p\leq\infty$ we have $a_k\leq k^{1/p}$, then there exists $C>0$ such that $a_n\leq Cn^{1/p}$ for each $n\in\N$.
\end{lemma}

The assertions contained in our next result have well-known analogies in the classical theory of Rademacher type/cotype (see, {\it e.g.}, \cite[Prop.~G.3, G.4]{benyamini}, \cite[Lemme~1.3, 2.3]{MP}, \cite[Lemma~13.5]{MS}). Clauses (vi) and (viii) correspond to the Maurey--Pisier theorem in its most particular form. They refer directly to Beauzamy-type results which we shall prove in the next section. Clauses (ii) and (iv) also rely on results which are not yet proved, so we postpone their proofs to Section~3.

\begin{proposition}\label{basic_prop}
Let $a\geq 0$, $b\geq 1$ and let $\mathcal{E}\subset \bigcup_{n=1}^\infty \mathcal{B}^{a,b}_n$ be an~identity block. 
\begin{enumerate}[{\rm (i)}]
\item If $1<N\in \nn$ and $1\leqslant p\leqslant r<\infty$ are such that $\alpha_{r,N}(\mathcal{E})\leq N^{\frac{1}{p}-\frac{1}{r}}$, then $\mathcal{E}$ has block basic type $s$ for every $1\leqslant s<p$.  

\item $\mathcal{E}$ has block basic type $p$ for some $1<p<\infty$ if and only if  every $(x_i)_{i=1}^\infty$ block finitely representable on $\mathcal{E}$ is shrinking.   

\item  If $1\leqslant r\leqslant q < \infty$ and $1<N\in \nn$ are such that $\beta_{r,N}(\mathcal{E})\leq N^{\frac{1}{r}-\frac{1}{q}}$, then $\mathcal{E}$ has block basic cotype $s$ for every $q<s<\infty$.  

\item $\mathcal{E}$ has block basic cotype $q$ for some $1<q<\infty$ if and only if every $(x_i)_{i=1}^\infty$ block finitely representable on $\mathcal{E}$ is boundedly complete.  

\item If $1<N\in \nn$ and $1\leqslant p\leqslant r<\infty$ are such that $\tau_{r,N}(\mathcal{E})\leq N^{\frac{1}{p}-\frac{1}{r}}$, then $\mathcal{E}$ has block type $s$ for every $1\leqslant s<p$.  

\item $\mathcal{E}$ has no nontrivial block  type if and only if the canonical $\ell_1$ basis is block finitely representable on $\mathcal{E}$. 

\item  If $1\leqslant r\leqslant q<\infty$ and $1<N\in \nn$ are such that $\gamma_{r,N}(\mathcal{E})\leq N^{\frac{1}{r}-\frac{1}{q}}$, then $\mathcal{E}$ has block cotype $s$ for every $q<s<\infty$. 

\item $\mathcal{E}$ fails to have nontrivial block cotype if and only if the canonical $c_0$ basis is block finitely representable on $\mathcal{E}$. 
\end{enumerate}
\end{proposition}

Before giving the proof of Proposition~\ref{basic_prop} let us show a~lemma which establishes duality between the~block basic type and the~block basic cotype. First, we introduce the following notation: For each basic sequence $(e_i)_{i=1}^n\in\eee$ ({\it cf.} the~remarks after Definition~\ref{def_collection}) we denote by $(e_i^\ast)_{i=1}^n$ its biorthogonal sequence, that is, the sequence of coordinate functionals. For any $n\in\N$, we define $\eee_n^\ast$ to be the closure in $\mathcal{B}^{a,b}_n$ of the family of all normalized block sequences of length $n$ formed from $\bigcup_{m=n}^\infty \{(e_i^\ast)_{i=1}^m\colon (e_i)_{i=1}^m\in\eee\}$. Let $\eee^{(\ast)}$ be the family of all normalized block sequences of all sequences that are biorthogonal to sequences in $\eee$. Note that, for each $n\in\N$, $\eee_n^\ast$ is the~closure of $\eee^{(\ast)}\cap \mathcal{B}_n^{a,b}$. Finally, we define $\eee^\ast=\bigcup_{n=1}^\infty \eee_n^\ast$.

\begin{lemma}\label{duality_lemma}
Let $a\geq 0$, $b\geq 1$ and let $\mathcal{E}\subset \bigcup_{n=1}^\infty \mathcal{B}^{a,b}_n$ be an~identity block. For every $1\leq p\leq\infty$ and $n\in\N$ we have
$$
\alpha_{p,n}(\eee)\leq\beta_{p^\prime,n}(\eee^\ast)\leq 2b\alpha_{p,n}(\eee)
$$
and
$$
\alpha_{p,n}(\eee^\ast)\leq\beta_{p^\prime,n}(\eee)\leq 2b\alpha_{p,n}(\eee^\ast).
$$
\end{lemma}
\begin{proof}
We start with proving the inequalities: $\alpha_{p,n}(\eee)\leq\beta_{p^\prime,n}(\eee^\ast)$ and $\beta_{p,n}(\eee)\leq 2b\alpha_{p^\prime,n}(\eee^\ast)$. Of course, the latter one gives the last inequality of our assertion because the roles of $p$ and $p^\prime$ are symmetric. 

Fix any $n\in\N$, $(e_i)_{i=1}^n \in \mathcal{E}$ and a~scalar sequence $(c_i)_{i=1}^n$. Pick a~norm one functional $x^\ast=\sum_{i=1}^n d_i' e_i^\ast$ with $x^\ast(\sum_{i=1}^n c_ie_i)=\n{\sum_{i=1}^n c_ie_i}$. We have 
\begin{equation*}
\begin{split}
1=\Biggl\|\sum_{i=1}^n d_i^\prime e_i^\ast\Biggr\| = \Biggl\|\sum_{i=1}^n \n{e_i^\ast}d_i^\prime \frac{e^\ast_i}{\|e_i^\ast\|}\Biggr\| &\geqslant \beta_{p^\prime, n}(\mathcal{E}^\ast)^{-1} \|(\|e_i^\ast\|d_i^\prime)_{i=1}\|_{\ell_{p^\prime}^n}\\
& \geqslant \beta_{p^\prime,n}(\mathcal{E}^\ast)^{-1} \|(d_i^\prime)_{i=1}^n\|_{\ell_{p^\prime}^n}.
\end{split}
\end{equation*}
Therefore,
$$
\Biggl\|\sum_{i=1}^n c_i e_i\Biggr\| = \sum_{i=1}^n c_id_i^\prime \leqslant \|(c_i)_{i=1}^n\|_{\ell_p^n} \|(d_i^\prime)_{i=1}^n\|_{\ell_{p^\prime}^n} \leqslant \beta_{p^\prime,n}(\mathcal{E}^\ast)\|(c_i)_{i=1}^n\|_{\ell_p^n}
$$
which gives the first of the announced inequalities.

For the second inequality, pick a scalar sequence $(d_i)_{i=1}^n$ so that $\|(d_i)_{i=1}^n\|_{\ell_{p^\prime}^n}=1$ and $\sum_{i=1}^n c_id_i=\|(c_i)_{i=1}^n\|_{\ell_p^n}$. Then we have 
$$
\Biggl\|\sum_{i=1}^n d_ie_i^\ast\Biggr\| = \Biggl\|\sum_{i=1}^n \|e_i^\ast\|d_i \frac{e^\ast_i}{\|e_i^\ast\|}\Biggr\| \leqslant \alpha_{p^\prime,n}(\mathcal{E}^\ast) \|(\|e_i^\ast\|d_i)_{i=1}\|_{\ell_{p^\prime}^n} \leqslant 2b\alpha_{p^\prime,n}(\mathcal{E}^*)
$$
and hence
$$
2b\alpha_{p^\prime,n}(\mathcal{E}^\ast)\Biggl\|\sum_{i=1}^n c_i e_i\Biggr\|\geqslant \Bigl(\sum_{i=1}^n d_ie_i^\ast\Bigr)\Bigl(\sum_{i=1}^n c_i e_i\Bigr)=\|(c_i)_{i=1}^n\|_{\ell_p^n}
$$
as desired.

By density, for proving the remaining two inequalities it is sufficient to consider only members of $\eee^{(\ast)}$. For each $(x_i^\ast)_{i=1}^n\in \mathcal{E}^{(\ast)}$ there exist $n\leqslant m\in\N$, $(e_i)_{i=1}^m\in\mathcal{E}$, integers $0=k_0<\ldots <k_n=m$ and a~scalar sequence $(b_i)_{i=1}^m$ such that $x_i^\ast=\sum_{k_{i-1}<j\leq k_j} b_j e_j^\ast$. Since $b$ is the basis constant, we have $1\leqslant \|e^*_j\|\leqslant 2b$ for each $1\leq j\leq m$.

Fix any scalar sequence $(c_i)_{i=1}^n$ and $(x_i^\ast)_{i=1}^n\in\eee$ as above. Pick a~norm one vector $x=\sum_{i=1}^m a_ie_i$ so that $\n{\sum_{i=1}^n c_ix_i^\ast}=(\sum_{i=1}^n c_ix_i^\ast)(x)$. Set 
$$
I=\Bigl\{i\leqslant n\colon \sum_{j=k_{i-1}+1}^{k_i} a_je_j\not=0\Bigr\}
$$
and define
$$
x_i^\prime=\left\{
\begin{array}{cl}
\displaystyle{\frac{\sum_{j=k_{i-1}+1}^{k_i} a_j e_j}{\left\|\sum_{j=k_{i-1}+1}^{k_i} a_j e_j\right\|}} & \mbox{if }i\in I\\
e_{k_i} & \mbox{if }i\not\in I,
\end{array}\right.\qquad 
d_i^\prime=\left\{
\begin{array}{cl}
\left\|\sum_{j=k_{i-1}+1}^{k_i} a_j e_j\right\| & \mbox{if }i\in I\\
0 & \mbox{if }i\not\in I.
\end{array}\right.
$$
Then
$$
1= \Biggl\|\sum_{i=1}^m a_i e_i\Biggr\| = \Biggl\|\sum_{i=1}^n d_i^\prime x_i^\prime\Biggr\| \geqslant \beta_{p,n}(\mathcal{E})^{-1} \|(d_i^\prime)_{i=1}^n\|_{\ell_p^n}
$$
and 
\begin{equation*}
\begin{split}
\Biggl\|\sum_{i=1}^n c_i x^\ast_i\Biggr\| &=\Bigl(\sum_{i=1}^n c_ix^\ast_i\Bigr)(x) =\sum_{i=1}^n c_id_i^\prime x^\ast_i(x_i^\prime)\\
&\leq\sum_{i=1}^n |c_id_i'| \leqslant \|(c_i)_{i=1}^n\|_{\ell_{p'}^n}\|(d_i')_{i=1}^n\|_{\ell_p^n} \leqslant \beta_{p,n}(\mathcal{E})\|(c_i)_{i=1}^n\|_{\ell_{p'}^n}
\end{split}
\end{equation*}
which proves that $\alpha_{p',n}(\eee^\ast)\leq\beta_{p,n}(\eee)$ (and the same holds true after swapping $p$ and $p'$).

For the remaining inequality, pick $(x_i)_{i=1}^n\in\eee$ such that $x_i^\ast(x_i)\geq 1/2b$. Note this can be done since basis constants of elements of $\eee$ are at most $b$. By the form of $x_i$'s we can also assume that each $x_i$ is a~combination of $e_j$'s with $k_{i-1}<j\leq k_i$. Take also a~scalar sequence $(d_i)_{i=1}^n$ with $\n{(d_i)_{i=1}^n}_{\ell_p^n}=1$ and $\sum_{i=1}^n c_id_i = \|(c_i)_{i=1}^n\|_{\ell_{p'}^n}$. Then $\|\sum_{i=1}^n d_i x_i\|\leqslant \alpha_{p,n}(\mathcal{E})$ and 
\begin{equation*}
\begin{split}
\Biggl\|\sum_{i=1}^n c_ix_i^\ast\Biggr\| &\geq \alpha_{p,n}(\eee)^{-1}\Bigl(\sum_{i=1}^n c_ix_i^\ast\Bigr)\Bigl(\sum_{i=1}^n d_ix_i\Bigr)=\alpha_{p,n}(\eee)^{-1}\sum_{i=1}^n c_id_ix_i^\ast(x_i)\\
&\geq (\alpha_{p,n}(\eee)2b)^{-1} \|(c_i)_{i=1}^n\|_{\ell_{p'}^n}
\end{split}
\end{equation*}
as required.
\end{proof}



\begin{proof}[Proof of Proposition \ref{basic_prop}\,{\rm (i), (iii), (v), (vii)}]
(i) This assertion follows easily from the proof of (v) by omitting the~expectation over sign choices.

(iii) In view of Lemma~\ref{duality_lemma}, we have
$$
\alpha_{r',N}(\eee^\ast)\leq\beta_{r,N}(\eee)\leq N^{\frac{1}{r}-\frac{1}{q}}=N^{\frac{1}{q'}-\frac{1}{r'}}.
$$
By assertion (i) (applied to $r^\prime$ in the place of $r$ and $q^\prime$ in the place of $p$), we infer that for every $s>q$ the~sequence $(\alpha_{s',n}(\eee^\ast))_{n=1}^\infty$ is bounded. By appealing to Lemma~\ref{duality_lemma} once again, we get
$$
\sup_n\beta_{s,n}(\eee)\leq 2b\sup_n \alpha_{s',n}(\eee^\ast)<\infty.
$$

(v) For $s=1$ the result is trivial, so assume $s>1$ and pick any number $t$ with $s<t<p$. Then we have $\tau_{r,N}(\eee)<N^{1/t-1/r}$ and therefore Proposition~\ref{submP} and Lemma~\ref{folklore} imply that there is a~constant $C>0$ such that $\tau_{r,n}(\eee)\leq Cn^{1/t-1/r}$ for each $n\in\N$. (Notice that the~sequence $(\tau_{r,n}(\eee))_{n=1}^\infty$ is increasing as $\eee$ is block closed). Define 
$$
\alpha=1-\frac{1}{t},\quad q=\frac{s}{s-1}\quad\mbox{ and }\quad \beta=q-\frac{1}{\alpha}
$$
and note that $\alpha>1-1/s=1/q$, hence $\beta>0$.

Fix any $\phi\in (0,1)$ and set 
$$
D=\Bigl(\phi^{-q}C^{\frac{1}{\alpha}}\sum_{j=1}^\infty \phi^{j\beta}\Bigr)^{\! 1/q}<\infty.
$$
Fix $n\in\N$, $(e_i)_{i=1}^n\in\eee$ and any scalar sequence $(a_i)_{i=1}^n$. Let $E=[e_i\colon 1\leq i\leq n]$. For each $\e=(\e_i)_{i=1}^n\in\{\pm 1\}^n$ pick a~functional $x_\e^\ast\in B_{E^\ast}$ with
$$
x_\e^\ast\Bigl(\sum_{i=1}^n a_i\e_ie_i\Bigr)=\Biggl\|\sum_{i=1}^n a_i\e_ie_i\Biggr\|.
$$
For every $j\in\N$ define
$$
S_j=\bigl\{i\leq n\colon \phi_j<\abs{\EEE x_\e^\ast(\e_ie_i)}\leq \phi^{j-1}\bigr\},
$$
where the expectation value is taken, as usual, with respect to the~uniform probability on $\{\pm 1\}^n$. For all $j\in\N$ and $i\in S_j$ we also pick a~unimodular scalar $\delta_i$ so that $\abs{\EEE x_\e^\ast(\e_i e_i)}=\EEE x_\e^\ast(\e_i\delta_i e_i)$. Then we have
$$
\abs{S_j}\phi^j\leq \sum_{i\in S_j}\EEE x_\e^\ast(\e_i\delta_ie_i)\leq\Biggl\|\sum_{i\in S_j}\e_i\delta_ie_i\Biggr\|_{L_r(E)}\!\!\leq C\abs{S_j}^{\frac{1}{t}}.
$$
Here, we have used the fact that $(\delta_ie_i)_{i\in S_j}\in\eee$ and $\tau_{r,\abs{S_j}}(\eee)\leq C\abs{S_j}^{1/t-1/r}$ if $S_j\not=\varnothing$, whereas all the~terms equal zero if $S_j=\varnothing$. Rearranging the~above inequality gives
$$
\abs{S_j}^\alpha \phi^j=\abs{S_j}^{1-\frac{1}{t}}\phi^j\leq C
$$
which implies that
$$
\abs{S_j}\phi^{jq}=\bigl(\abs{S_j}^\alpha\phi^j\bigr)^{\frac{1}{\alpha}}\phi^{j\beta}\leq C^{\frac{1}{\alpha}}\phi^{j\beta}.
$$

Now, define 
$$
x^\ast=\sum_{j=1}^\infty \phi^{j-1}\sum_{i\in S_j} e_i^\ast\in\ell_q
$$
($e_i^\ast$ stands here for the $i$th element of the~canonical basis of $\ell_q$) and observe that since the sets $S_j$ are mutually disjoint, we have
$$
\n{x^\ast}_{\ell_q}=\Bigl(\sum_{j=1}^\infty \phi^{-q}\phi^{jq}\abs{S_j}\Bigr)^{\!\! 1/q}\leq \Bigl(\phi^{-q}C^{\frac{1}{\alpha}}\sum_{j=1}^\infty \phi^{j\beta}\Bigr)^{\!\! 1/q}=D.
$$
Hence, if $B_s$ stands for the constant stemming from the Kahane--Khintchine inequality (see, {\it e.g.} \cite[\S 9.2]{MS}), then we have
\begin{equation*}
\begin{split}
\frac{1}{B_s}\Biggl\|\sum_{i=1}^n a_i\e_ie_i\Biggr\|_{L_s(E)}\!\!\! & \leq \Biggl\|\sum_{i=1}^n a_i\e_ie_i\Biggr\|_{L_1(E)}\\
&=\EEE x_\e^\ast\Bigl(\sum_{i=1}^n a_i\e_ie_i\Bigr)\leq \sum_{i=1}^n\abs{a_i}\abs{\EEE x_\e^\ast(\e_ie_i)}\\
&=\sum_{j=1}^\infty\sum_{i\in S_j}\abs{a_i}\abs{\EEE x_\e^\ast(\e_ie_i)}\leq \sum_{j=1}^\infty\sum_{i\in S_j}\abs{a_i}\phi^{j-1}\\
&=x^\ast((a_i)_{i=1}^n)\leq \n{x^\ast}_{\ell_q}\n{(a_i)_{i=1}^n}_{\ell_s^n}\leq D\n{(a_i)_{i=1}^n}_{\ell_s^n}.
\end{split}
\end{equation*}
Consequently, $\eee$ has block type $s$ with constant at most $DB_s$.

(vii) Proposition~\ref{submP} and Lemma~\ref{folklore} imply that there is a~constant $C>0$ such that $\gamma_{r,n}(\eee)\leq Cn^{1/r-1/q}$ for every $n\in\N$. Fix any $\phi\in (0,1)$ and set
$$
D=\phi^{-1}\sum_{j=1}^\infty \phi^{j(1-\frac{q}{s})}<\infty.
$$
Fix $n\in\N$, $(e_i)_{i=1}^n\in\eee$ and any scalar sequence $(a_i)_{i=1}^n$ such that $\max_{1\leq i\leq n}\abs{a_i}=1$. For each $1\leq i\leq n$ pick a~unimodular scalar $\delta_i$ so that $\delta_ia_i=\abs{a_i}$. Replacing $(a_i)_{i=1}^n$ by $(\delta_ia_i)_{i=1}^n$ and $(e_i)_{i=1}^n$ by $(\delta_i^{-1} e_i)_{i=1}^n\in \mathcal{E}$ we may assume each $a_i$ is nonnegative. Let $E=[e_i\colon 1\leq i\leq n]$. For each $j\in \nn$ let 
$$
S_j=\{ i\leqslant n: \phi^j<|a_i|\leqslant \phi^{j-1}\};
$$
note that 
$$
\|(a_i)_{i=1}^n\|_{\ell_s^n}  \leqslant \phi^{-1} \sum_{j=1}^\infty \phi^j|S_j|^{1/s}.
$$
For every $j\in \nn$ with $S_j\neq \varnothing$ we have
\begin{equation*}
\begin{split}
\gamma_{r, |S_j|}(\mathcal{E}) \Biggl\|\sum_{i=1}^n \ee_i a_i e_i\Biggr\|_{L_s(E)} &\geq \gamma_{r, |S_j|} \Biggl\|\sum_{i=1}^n \ee_i a_i e_i\Biggr\|_{L_r(E)}\\
&\geq \phi^j \gamma_{r, |S_j|} \Biggl\|\sum_{i\in S_j} \ee_i e_i\Biggr\|_{L_r(E)}\geq \phi^j |S_j|^{1/r}.  
\end{split}
\end{equation*}
Here we have used the fact that $a_i$'s are real numbers and that the~sequence $(\ee_i e_i)_{i=1}^n$ is $1$-unconditional in $L_r(E)$. Therefore, we get
$$
\phi^j |S_j|^{1/r} \leq\gamma_{r, |S_j|} \Biggl\|\sum_{i=1}^n \ee_i a_i e_i\Biggr\|_{L_s(E)}\leqslant C|S_j|^{\frac{1}{r}-\frac{1}{q}},
$$ 
thus 
$$
|S_j|^\frac{1}{q} \leqslant C\phi^{-j}\Biggl\|\sum_{i=1}^n \ee_i a_i e_i\Biggr\|_{L_s(E)}.
$$
Note that this last inequality holds also for $j\in \nn$ with $S_j=\varnothing$, so it is true for all $j\in\N$.

Now, since $\max_{1\leqslant i\leqslant n}|a_i|=1$ and $q/s<1$, we obtain
$$
|S_j|^{\frac{1}{s}} \leq C^{\frac{q}{s}} \phi^{-\frac{jq}{s}} \Biggl\|\sum_{i=1}^n \ee_i a_i e_i\Biggr\|_{L_s(E)}^{\frac{q}{s}} \leq C \phi^{-\frac{jq}{s}} \Biggl\|\sum_{i=1}^n \ee_i a_i e_i\Biggr\|_{L_s(E)},
$$
hence
$$
\|(a_i)_{i=1}^n\|_{\ell_s^n} \leq \phi^{-1}\sum_{i=1}^n \phi^j |S_j|^{1/s} \leqslant CD \Biggl\|\sum_{i=1}^n \ee_i a_i e_i\Biggr\|_{L_s(E)}.$$
This shows that $\gamma_{s,n}(\mathcal{E})\leqslant CD$ for each $n\in \nn$ and the assertion follows.
\end{proof}

The above proof contains the following `automatic-type' result which was used several times along the way and which is worth recording separately.

\begin{corollary}\label{aut_type}
If $\eee\subset \bigcup_{n=1}^\infty \mathcal{B}^{a,b}_n$ is an identity block, then $\eee$ has block (basic) subtype/subcotype if and only if it has nontrivial block (basic) type/cotype.
\end{corollary}

\section{Beauzamy-type results}
It is a classical fact, and a~corollary from the~Maurey--Pisier theorem \cite{MP}, that $\ell_1$ (respectively: $c_0$) is crudely finitely representable in a~Banach space $X$ if and only if $X$ fails to have nontrivial Rademacher type (resp. cotype). This does not extend verbatim to operators as is easily seen by considering a~diagonal operator (on $\ell_1$ or $c_0$) with diagonal entries vanishing like $1/\log n$. Nevertheless, an~appropriate operator version of this phenomenon was given by Beauzamy \cite{Beau} who proved that $\ell_1$ (resp. $c_0$) is not  crudely finitely representable in an~operator $A$ if and only if $A$ has Rademacher subtype (resp. subcotype),  which is actually equivalent to $A$ having subtype $2$ (resp. subcotype $2$). The aim of the present section is to derive analogous results in the~language of general block structures.
\begin{definition}
Given $a\geqslant 0$, $b\geqslant 1$, $n\in \nn$ and $\varrho>0$, we say that $(r, \rho)\in \mathcal{B}^{a,b}_n$ is:
\begin{enumerate}[(i)]
\item a {\it $\varrho$-$\ell_1^+$-sequence} if for every sequence $(a_i)_{i=1}^n\subset [0,\infty)$ we have $$\rho\Bigl(\sum_{i=1}^n a_i \rme_i\Bigr) \geqslant \varrho\sum_{i=1}^n a_i,$$   
\item a {\it $\varrho$-$\ell_1$-sequence} if for every scalar sequence $(a_i)_{i=1}^n$ we have $$\rho\Bigl(\sum_{i=1}^n a_i \rme_i\Bigr) \geqslant \varrho \sum_{i=1}^n |a_i|,$$ 
\item a {\it $\varrho$-$c_0^+$-sequence} if there exist $(a_i)_{i=1}^n\subset [\varrho, \infty)$ such that $\rho(\rme_i)\geqslant \varrho$ for each $1\leqslant i\leqslant n$ and $$\max_{1\leqslant m\leqslant n} r\Bigl(\sum_{i=1}^m  a_i \rme_i\Bigr)\leqslant 1,$$ 
\item a {\it $\varrho$-$c_0$-sequence} if $\rho(\rme_i)\geqslant \varrho$ for each $1\leqslant i\leqslant n$ and $$\max_{(a_i)_{i=1}^n \in B_{\ell_\infty^n}}\! \varrho r\Bigl(\sum_{i=1}^n a_i \rme_i\Bigr) \leqslant 1.$$  
\end{enumerate} 
\end{definition}

\begin{theorem}\label{bbtype_thm} 
Let $a\geqslant 0$, $b\geqslant 1$ and $\mathcal{E}\subset \bigcup_{n=1}^\infty \mathcal{B}^{a,b}_n$ be a~stable set. Then, the following assertions are equivalent:
\begin{enumerate}[{\rm (i)}]
\item $\mathcal{E}$ has block basic subtype $p$ for every $1<p<\infty$. 
\item $\mathcal{E}$ has block basic subtype $p$ for some $1<p<\infty$. 
\item $\mathcal{E}$ has block basic subtype. 
\item For every $\varrho>0$ there exists $N=N(\varrho)\in \nn$ such that $\mathcal{E}\cap \mathcal{B}^{a,b}_n$ contains no $\varrho$-$\ell_1^+$-sequences whenever $n\geq N$.
\item Every operator which is block finitely representable on $\eee$ is shrinking.
\end{enumerate}
\end{theorem}
\begin{proof}
(i) $\Rightarrow$ (ii) is trivial.

(ii) $\Rightarrow$ (v) Let $1<p<\infty$ be such that $\sup_n\alpha_{p,n}(\mathcal{E})$ is finite and let $A\colon[x_i\colon i\in \nn]\to [Ax_i\colon i\in \nn]$ be block finitely representable on $\mathcal{E}$. If $(z_i)_{i=1}^\infty$ is any normalized block subsequence of $(x_i)_{i=1}^\infty$, then 
$$
\Biggl\|\frac{1}{n}A\sum_{i=1}^n z_i\Biggr\|\leqslant \frac{\alpha_{p,n}(\mathcal{E})}{n^{1/p}}\xrightarrow[\, n\to\infty\,]{} 0.
$$ 
This means that every normalized block subsequence of $(x_i)_{i=1}^\infty$ has a~convex block subsequence whose image under $A$ converges to zero in norm. Hence, $A$ is shrinking.

(v) $\Rightarrow$ (iv) We argue by contraposition. Assume there is $\varrho>0$ such that for every $n\in\N$ there exists a~$\varrho$-$\ell_1^+$-sequence $(r_n,\rho_n)\in \eee\cap\mathcal{B}^{a,b}_n$. Fix a~free ultrafilter $\mathcal{U}$ of subsets of $\N$. In the~linear space $c_{00}(\mathbb{K})$ of all finitely supported sequences in $\mathbb{K}$ define a~norm $r$ and a~seminorm $\rho$ by the formulas:
$$
r\Bigl(\sum_{i=1}^k a_i\rme_i\Bigr)=\mathcal{U}\mbox{-}\lim_{n}r_n\Bigl(\sum_{i=1}^k a_i\rme_i\Bigr)\quad\mbox{ and }\quad \rho\Bigl(\sum_{i=1}^k a_i\rme_i\Bigr)=\mathcal{U}\mbox{-}\lim_{n}\rho_n\Bigl(\sum_{i=1}^k a_i\rme_i\Bigr).
$$
Set $Z=\{x\in c_{00}(\mathbb{K})\colon \rho(x)=0\}$. This is a~linear subspace of $c_{00}(\mathbb{K})$ and the formula $r'(x+Z)=\rho(x)$ defines a~norm on the~quotient space $c_{00}(\mathbb{K})/Z$. Consider the operator $\widetilde{A}\colon c_{00}(\mathbb{K})\to c_{00}(\mathbb{K})/Z$ given by $\widetilde{A}(\rme_i)=\rme_i+Z$ for $i\in\N$; it is continuous with respect to the norms $r$ and $r^\prime$. Let $A$ be its unique continuous extension to $c_0(\mathbb{K})$, the completion of $c_{00}(\mathbb{K})$ with respect to $r$. Since $\eee$ is closed, $A$ is block finitely representable on $\eee$. Note also that for every sequence $(a_i)_{i=1}^k\subset [0,\infty)$ we have 
$$
r^\prime\Bigl(A\sum_{i=1}^k a_i \rme_i\Bigr)=\rho\Bigl(\sum_{i=1}^k a_i\rme_i\Bigr)\geq \varrho\sum_{i=1}^k a_i
$$
which shows that $A$ is not shrinking and therefore condition (v) fails to hold.

(iv) $\Rightarrow$ (i) Fix arbitrarily $\varrho>0$ and pick $N\in\N$ so that $\eee\cap\mathcal{B}^{a,b}_n$ contains no $\varrho/2$-$\ell_1^+$-sequences whenever $n\geq N$. For an~arbitrary $p\in (1,\infty)$ choose $M\in\N$ such that $M\geq N$ and $aN/M^{1/p^\prime}<\varrho/2$.

Fix any $n\geq M$, $(r,\rho)\in \eee\cap\mathcal{B}^{a,b}_n$ and a~scalar sequence $(a_i)_{i=1}^n\in B_{\ell_p^n}$. We may find a~functional $f^\ast\in B_{(\mathbb{K}^n, \rho)^\ast}$ such that 
$$
\mathrm{Re}\,f^\ast\Bigl(\sum_{i=1}^n a_i \rme_i\Bigr)= \rho\Bigl(\sum_{i=1}^n a_i \rme_i\Bigr).
$$
Note that $(\mathbb{K}^n, \rho)$ is only assumed to be a~seminormed space, however, we may pass first to the normed space $([\rme_i\colon i\in I], \rho)$, where $I=\{i\colon \rho(\rme_i)\neq 0\}$, choose a~norming functional for $\sum_{i\in I} a_i \rme_i$ and then extend it by $0$ to $\mathbb{K}^n$.  

Let $J=\{i\in I\colon |f^\ast(\rme_i)|\geqslant \varrho/2\}$. Notice that $|J|\leqslant N$, as otherwise we could choose unimodular scalars $(\delta_j)_{j\in J}$ such that $(\delta_j \rme_j, \delta_j \rme_j)_{j\in J}\subset (\mathbb{K}^n, r)\times (\mathbb{K}^n, \rho)$ forms a~$\varrho/2$-$\ell_1^+$-sequence in $\mathcal{E}$. Then 
\begin{equation*}
\begin{split}
\rho\Bigl(\sum_{i=1}^n a_i e_i\Bigr) &\leq \sum_{j\in J} a + \|(a_i)_{i\in I\setminus J}\|_{\ell_p^{|I|-|J|}}\|(|f^\ast(\rme_i)|)_{i\in I\setminus J}\|_{\ell_{p^\prime}^{|I|-|J|}}\\
&\leq a N + (\varrho/2) n^{1/p^\prime} < \varrho n^{1/p^\prime}. 
\end{split}
\end{equation*} 
Since $\varrho>0$ was arbitrary, we obtain $\alpha_{p,n}(\mathcal{E})=o(n^{1/p^\prime})$, that is, $\eee$ has block basic subtype $p$.

(iv) $\Rightarrow$ (iii) can be shown similarly. Indeed, we fix $\varrho>0$ and $N\in\N$ as above and we pick $M\geq N$ so that $aN/M<\varrho/2$. Then we deduce as in the~previous paragraph that for every $n\geq M$ we have $\rho(\frac{1}{n}\sum_{i=1}^n \rme_i)<\varrho$.

(iii) $\Rightarrow$ (iv) is obvious and the proof has been completed. 
\end{proof}

The proof of our next result is split in three parts. We separately prove that conditions (i)--(iii) and (iv)--(vi) are equivalent, while the~equivalence between (iii) and (iv) follows from a~modification of the~work of Beauzamy \cite[Th\'{e}or\`{e}me~1]{Beau}. His argument concerns arbitrary vectors in a~Banach space, as opposed to members of a~prescribed set of sequences. Nonetheless, the techniques used involve only passing to block sequences and therefore can be repeated under our hypotheses.
\begin{theorem}\label{btype_thm}
Let $a\geqslant 0$, $b\geqslant 1$ and $\mathcal{E}\subset \bigcup_{n=1}^\infty \mathcal{B}^{a,b}_n$ be a~stable set. Then, the following assertions are equivalent:
\begin{enumerate}[{\rm (i)}]
\item $\mathcal{E}$ has block subtype $p$ for every $1<p<\infty$. 
\item $\mathcal{E}$ has block subtype $p$ for some $1<p<\infty$. 
\item $\mathcal{E}$ has block subtype.
\item For every $\varrho>0$ there exists $N=N(\varrho)\in \nn$ such that $\mathcal{E}\cap \mathcal{B}^{a,b}_n$ contains no $\varrho$-$\ell_1$-sequences whenever $n\geq N$. 
\item There do not exist any renormings $|\!\cdot\!|_1$ and $|\!\cdot\!|_2$ of $\ell_1$ such that the formal identity $\iota\colon (\ell_1, |\cdot|_1)\to (\ell_1, |\cdot|_2)$ is block finitely representable on $\eee$. 
\item There does not exist $c>0$ such that $c\cdot\mathrm{Id}_{\ell_1}$ is block finitely representable on $\mathcal{E}$.  
\end{enumerate}
\end{theorem}
\begin{proof}
We first argue that (i) $\Rightarrow$ (ii) $\Rightarrow$ (iii) $\Rightarrow$  (i). The first implication is trivial. Since $\Theta_n(\mathcal{E})\leq\tau_{p,n}(\mathcal{E})/n^{1/p^\prime}$ for any $1<p<\infty$, we obtain that (ii) $\Rightarrow$ (iii). 

Now, suppose that (i) fails. Then we may pick $1<p<\infty$, $\vartheta$ and $\theta$ such that
\begin{equation}\label{limsup_B}
\limsup_{n\to\infty}\frac{\tau_{p,n}(\eee)}{n^{1/p^\prime}}>\vartheta>\theta>0.
\end{equation}
By the Kahane--Khintchine inequality, there exists a~constant $A_p>0$ such that for all $n\in \nn$, $(r, \rho)\in \mathcal{E}\cap \mathcal{B}^{a,b}_n$ and all scalars $(a_i)_{i=1}^n$ we have
$$
\Biggl\|\rho\Bigl(\sum_{i=1}^n a_i \ee_i \rme_i\Bigr)\Biggr\|_{L_p} \leq A_p \Biggl\|\rho\Bigl(\sum_{i=1}^n a_i \ee_i \rme_i\Bigr)\Biggr\|_{L_1}.
$$
Fix any $k,N\in \nn$ satisfying 
\begin{equation}\label{theta_B}
(\vartheta-\theta)N^{1/p^\prime}>aA_p k.
\end{equation}
In view of \eqref{limsup_B}, there is $n\geqslant N$, $(a_i)_{i=1}^n\in B_{\ell_p^n}$ and $(r, \rho)\in \mathcal{E}\cap \mathcal{B}^{a,b}_n$ such that 
$$
A_p\Biggl\|\rho\Bigl(\sum_{i=1}^n a_i \ee_i \rme_i\Bigr)\Biggr\|_{L_1} \geq \Biggl\|\rho\Bigl(\sum_{i=1}^n a_i \ee_i\rme_i\Bigr)\Biggr\|_{L_p} >\vartheta n^{1/p^\prime}.
$$
For every sign choice $\ee=(\ee_i)_{i=1}^n\in\{\pm 1\}^n$ pick a~norm one functional $f^\ast_\ee\in (\mathbb{K}^n, \rho)^\ast$ such that 
$$
\mathrm{Re}\,f^\ast_\ee\Bigl(\sum_{i=1}^n a_i \ee_i \rme_i\Bigr) =\rho\Bigl(\sum_{i=1}^n a_i \ee_i \rme_i\Bigr).
$$
Define
$$
F(i)= \mathbb{E}\,\bigl\{\mathrm{Re}\,f^\ast_\ee(\ee_i \rme_i)\bigr\}\quad\mbox{ and }\quad S=\bigl\{i\leqslant n\colon \abs{F(i)}\geq \theta/A_p\bigr\}
$$
(the expectation is taken over $\e\in\{\pm 1\}^n$). Then we have
\begin{equation*}
\begin{split}
\vartheta n^{1/p^\prime}/A_p & \leq \Biggl\|\rho\Bigl(\sum_{i=1}^n a_i \ee_i \rme_i\Bigr)\Biggr\|_{L_1}=\sum_{i=1}^n a_i F(i)\\
&\leq\sum_{i\in S} a + \sum_{i\not\in S}\theta\abs{a_i}/A_p\leq a\abs{S}+\theta n^{1/p^\prime}/A_p 
\end{split}
\end{equation*} 
and therefore inequality \eqref{theta_B} yields that $\abs{S}>k$.    

Let $T$ be a subset of $S$ with cardinality $k$ and for each $i\in T$ choose $\delta_i\in \{\pm 1\}$ with $\delta_i F(i)=\abs{F(i)}$. Then
$$
\Biggl\|\rho\Bigl(\sum_{i\in T} \frac{\ee_i }{k} \delta_i \rme_i\Bigr)\Biggr\|_{L_1(2^{\abs{T}})}\!\!\geq \mathbb{E}\Biggl\{\mathrm{Re}\,f_\ee^\ast\Bigl(\sum_{i\in T} \frac{\ee_i}{k} \delta_i \rme_i\Bigr)\Biggr\}=\sum_{i\in T}\frac{\abs{F(i)}}{k} \geq \theta/A_p.
$$
This shows that 
$$
\liminf_{n\to\infty}\Theta_n(\mathcal{E})\geq\limsup_{n\to\infty}\frac{\tau_{p,n}(\mathcal{E})}{n^{1/p^\prime}A_p}$$
and, in view of \eqref{limsup_B}, condition (iii) fails to hold. This finishes the first part of the proof.

It is quite obvious that (iii) $\Rightarrow$ (iv). For the converse we appeal to the~Beauzamy's result \cite[Th\'{e}or\`{e}me 1]{Beau} as explained before.

It remains to show that (iv) $\Rightarrow$ (vi) $\Rightarrow$ (v) $\Rightarrow$ (iv). The first implication is clear. To see that $\neg$(v) $\Rightarrow$ $\neg$(vi), we use the~classical James' non-distortion technique. 

Suppose $\abs{\,\cdot\,}_1$ and $\abs{\,\cdot\,}_2$ are two renormings of $\ell_1$ such that the formal identity between $(\ell_1,\abs{\,\cdot\,}_1)$ and $(\ell_1,\abs{\,\cdot\,}_2)$ is block finitely representable on $\eee$. The~canonical basis $(e_i)_{i=1}^\infty$ of $\ell_1$ equipped with each of the norms $\abs{\,\cdot\,}_1$ and $\abs{\,\cdot\,}_2$ is, of course, equivalent to the~original canonical $\ell_1$-basis. For convenience, we denote these two copies of $(e_i)_{i=1}^\infty$ by $(f_i)_{i=1}^\infty$ and $(g_i)_{i=1}^\infty$.

Take closed intervals $I,J\subset [0,\infty)$ such that for every scalar sequence $(a_i)_{i=1}^\infty\in S_{\ell_1}$ we have $\abs{\sum_{i=1}^\infty a_i f_i}_1\in I$ and $\abs{\sum_{i=1}^\infty a_i g_i}_2\in J$. By the James dichotomy (see {\it e.g.} \cite[Prop.~2]{OS_handbook}), either there exists a~subsequence $(f_i')_{i=1}^\infty$ of $(f_i)_{i=1}^\infty$ such that
$$
\Biggl|\sum_{i=1}^\infty a_i f_i'\Biggr|_1\in [\sqrt{\min I \max I},\max I]\quad\mbox{for every scalar sequence }(a_i)_{i=1}^\infty\in S_{\ell_1},
$$
or there exists a~block subsequence $(f_i')_{i=1}^\infty$ of $(f_i)_{i=1}^\infty$ such that for each $i\in \nn$ there are $L_i\subset \nn$ and scalars $(b_j)_{j\in L_i}$ satisfying:
$$
\sum_{j\in L_i} \abs{b_j}=1,\quad  f_i'=\sum_{j\in L_i} b_j f_j\quad\mbox{and}\quad \abs{f_i'}_1\in [\min I, \sqrt{\min I\max I}]\,\,\mbox{ for each }i\in\N.
$$
In the former case, we let $(g_i')_{i=1}^\infty$ be the corresponding subsequence of $(g_i)_{i=1}^\infty$ and define $I_1=[\sqrt{\min I\max I},\max I]$. In the latter case, we let $(g_i')_{i=1}^\infty$ be the corresponding block subsequence of $(g_i)_{i=1}^\infty$ and define $I_1=[\min I, \sqrt{\min I\max I}]$. In either case we have $\abs{\sum_{i=1}^\infty a_i g_i'}_2\in J$ for each $(a_i)_{i=1}^\infty\in S_{\ell_1}$. Applying the dichotomy again, this time to $(g_i')_{i=1}^\infty$, we obtain new sequences $(f_i^{(1)})_{i=1}^\infty$, $(g_i^{(1)})_{i=1}^\infty$ and a~subinterval $J_1\subset J$ such that:
$$
\log \frac{\max J_1}{\min J_1}=\frac{1}{2}\log \frac{\max J}{\min J},\quad \Biggl|\sum_{i=1}^\infty a_i f_i^{(1)}\Biggr|_1\in I_1 \,\,\,\mbox{and}\quad \Biggl|\sum_{i=1}^\infty a_i g_i^{(1)}\Biggr|_2\in J_1
$$
for all $(a_i)_{i=1}^\infty \in S_{\ell_1}$.

An $n$-fold application of this pair of dichotomies produces subintervals $I_n\subset I$, $J_n\subset J$ and a~blocking $(f_i^{(n)})_{i=1}^\infty$ of $(f_i)_{i=1}^\infty$ such that if $(g_i^{(n)})_{i=1}^\infty$ denotes the~corresponding blocking of $(g_i)_{i=1}^\infty$, then
$$
\log \frac{\max I_n}{\min I_n}= \frac{1}{2^n} \log \frac{\max I}{\min I}\quad\mbox{and}\quad\log \frac{\max J_n}{\min J_n}= \frac{1}{2^n} \log \frac{\max J}{\min J},
$$
and 
$$
\Biggl|\sum_{i=1}^\infty a_i f_i^{(n)}\Biggr|_1\in I_n,\quad \Biggl|\sum_{i=1}^\infty a_i g^{(n)}_i\Biggr|_2\in J_n\quad\mbox{for each }(a_i)_{i=1}^\infty \in S_{\ell_1}.
$$
Now, assuming $\neg$(v), we may pick $(f_i)_{i=1}^\infty$ and $(g_i)_{i=1}^\infty$ above so that the formal identity $[f_i\colon i\in\N]\to [g_i\colon i\in\N]$ belongs to $\eee$. Since $\eee$ is block closed, for all $k,n\in\N$ the~formal identity
$$
\iota_n\colon \Biggl[\frac{f^{(n)}_i}{|f_i^{(n)}|_1}\colon i\leqslant k\Biggr]\xrightarrow[\phantom{xxx}]{} \Biggl[\frac{g^{(n)}_i}{|g^{(n)}_i|_2}\colon i\leqslant k\Biggr]\in \mathcal{E}.
$$
Observe that $\iota_n$ converges pointwise to $c\cdot\mathrm{Id}_{\ell_1^k}$ as $n\to\infty$, where $\{c_1\}=\bigcap_n I_n$, $\{c_2\}=\bigcap_n J_n$ and $c=c_2/c_1$.  

Finally, we prove $\neg$ (iv) $\Rightarrow$ $\neg$(v) by constructing $\iota\colon (\ell_1,|\!\cdot\!|_1)\to (\ell_1, |\!\cdot\!|_2)$ with the aid of arbitrarily long $\varrho$-$\ell_1$-sequences in $\mathcal{E}$ and an~ultrafilter as in the proof of Theorem~\ref{bbtype_thm}. 
\end{proof}

\begin{theorem}\label{bbco_thm}
Let $a\geqslant 0$, $b\geqslant 1$ and $\mathcal{E}\subset \bigcup_{n=1}^\infty \mathcal{B}^{a,b}_n$ be a~stable set. Then, the following assertions are equivalent:
\begin{enumerate}[{\rm (i)}]
\item $\mathcal{E}$ has block basic subcotype $q$ for every $1<q<\infty$. 
\item $\mathcal{E}$ has block basic subcotype $q$ for some $1<q<\infty$. 
\item $\mathcal{E}$ has block basic subcotype.
\item For every $\varrho>0$ there exists $N=N(\varrho)\in \nn$ such that $\mathcal{E}\cap \mathcal{B}^{a,b}_n$ contains no $\varrho$-$c_0^+$-sequences whenever $n\geq N$.
\item Every operator which is block finitely representable on $\eee$ is boundedly complete.
\end{enumerate}
\end{theorem}
\begin{proof}
(i) $\Rightarrow$ (ii) is trivial.  

(ii) $\Rightarrow$ (iii) Fix $(r, \rho)\in \mathcal{E}\cap \mathcal{B}^{a,b}_n$ and assume $(a_i)_{i=1}^n$ is a~scalar sequence with $\abs{a_i \rho(e_i)}\geqslant 1$ for each $1\leq i\leq n$. Then $\beta_{n,q}(\mathcal{E})r(\sum_{i=1}^n a_i e_i)\geq n^{1/q}$, hence $\upsilon_{n}(\mathcal{E})\leq \beta_{q,n}(\mathcal{E})/n^{1/q}\xrightarrow[]{\phantom{xx}} 0$. 

\hspace*{-3mm}$\neg$(i) $\Rightarrow$ $\neg$(iii) Assume that $1<q<\infty$ is such that 
\begin{equation}\label{subco1}
\limsup_{n\to\infty}\frac{\beta_{q,n}(\mathcal{E})}{n^{1/q}}>\vartheta>0.
\end{equation}
Fix any $n\in \nn$ and $c>1$. Let $C=2ab/\vartheta$ and pick $N\geq n$ such that 
$$
(1-1/c^q)N> C^q n.
$$
According to \eqref{subco1} we may take $M\geqslant N$, a~ scalar sequence $(a_i)_{i=1}^M$ and $(r,\rho)\in \mathcal{E}\cap \mathcal{B}^{a,b}_M$ such that $\n{(a_i \rho(\rme_i))_{i=1}^M}_{\ell_q^M}=M^{1/q}$ and $r(\sum_{i=1}^M a_i \rme_i)<1/\vartheta$. Observe that 
$$
\max_{1\leqslant i\leqslant M} |a_i\rho(\rme_i)| \leqslant a\max_{1\leqslant i\leqslant M} r(a_i\rme_i)\leqslant 2abr\Big(\sum_{i=1}^M a_i \rme_i\Big)<C.
$$  
Let $S=\{i\leqslant M\colon |a_i \rho(\rme_i)|\geqslant 1/c\}$ and note that 
$$
M=\sum_{i=1}^M |a_i \rho(\rme_i)|^q \leqslant C^q |S|+ \sum_{i\not\in S} |a_i\rho(\rme_i)|^q \leqslant C^q |S|+M/c^q
$$
and hence $|S|>M\geq n$. 

Let $I_1,\ldots, I_n$ be a~partition of $\{1,\ldots,M\}$ into intervals such that for each $1\leqslant i\leqslant n$ we have $I_i\cap S\neq \varnothing$. For each $1\leq i\leq n$ define $f_i=\sum_{j\in I_i} a_j \rme_j$ and notice that if $j\in I_i\cap S$, then
$$
1/c\leqslant |a_j\rho(\rme_j)| \leqslant 2b\rho(f_i).
$$
Now, let $g_i=f_i/r(f_i)$ and $b_i=2bc r(f_i)$ for $i\in\N$. Then 
$$
|b_i\rho(g_i)| = 2bc \rho(f_i)\geqslant 1
$$
and
$$r\Big(\sum_{i=1}^n b_i g_i\Big)= 2bc r\Big(\sum_{i=1}^M a_i \rme_i\Big)< \frac{2bc}{\vartheta}.
$$
Since $c>1$ was arbitrary, it follows that $\upsilon_n(\mathcal{E}) \geq\vartheta/2b$ for any $n\in \nn$, thus (iii) does not hold true.

(i) $\Rightarrow$ (v) can be shown by repeating the~argument from the proof of Proposition~\ref{basic_prop}(iv).

$\neg$(iv) $\Rightarrow$ $\neg$(v) Take $\varrho>0$ such that for each $n\in \nn$, $\mathcal{E}\cap \mathcal{B}^{a,b}_n$ admits a~$\varrho$-$c_0^+$-sequence, say $(r_n, \rho_n)$.  For $n\in \nn$ pick a~scalar sequence $(a^n_i)_{i=1}^n\subset [\varrho, \infty)$ such that $\rho_n(\rme_i)\geqslant \varrho$ for each $1\leqslant i\leqslant n$ and $r_n(\sum_{i=1}^m a_i^n \rme_i)\leqslant 1$ for each $1\leq m\leq n$. Note that $|a^n_i|\leqslant 2b$ for all $n\in \nn$, $1\leqslant i\leqslant n$. Fix a~free ultrafilter $\mathcal{U}$ on $\nn$ and consider the norm $r$ and the seminorm $\rho$ on $c_{00}(\mathbb{K})$ defined as in the proof of the~implication (v) $\Rightarrow$ (iv) in Theorem~\ref{bbtype_thm}. Define $a_i=\mathcal{U}\mbox{-}\lim_n a^n_i$ for $i\in\N$ and let $A\colon [\rme_i\colon i\in \nn]\to [A\rme_i\colon i\in \nn]$ be the~continuous extension of the~operator $\widetilde{A}\colon c_{00}(\mathbb{K})\to c_{00}(\mathbb{K})/Z$ given by $\widetilde{A}(\rme_i)=a_i\rme_i+Z$, where $Z=\{x\in c_{00}(\mathbb{K})\colon \rho(x)=0\}$. Then, $A$ is block finitely representable on $\eee$ but evidently fails to be boundedly complete.

(iv) $\Rightarrow$ (iii) Fix $\varrho>0$ and take $N=N(\varrho)\in \nn$ such that for any $n\geqslant N$, $\mathcal{E}\cap \mathcal{B}^{a,b}_n$ admits no $\varrho$-$c_0^+$-sequences. Assume that for some $n\geqslant N$, $C>0$, $(r, \rho)\in \mathcal{E}\cap \mathcal{B}^{a,b}_n$ and $(a_i)_{i=1}^n$ are such that $|a_i \rho(\rme_i)|\geqslant 1$ for each $1\leqslant i\leqslant n$ and $r(\sum_{i=1}^n a_i \rme_i)\leqslant C$. Then $\max_{1\leqslant m\leqslant n}|a_m|, \max_{1\leqslant m\leqslant n}r(\sum_{i=1}^m a_i \rme_i) \leqslant 2bC$, thus $\rho(\rme_i)\geq 1/(2bC)$. If $2bC\leqslant 1/\varrho$, then $(r, \rho)$ is a~$\varrho$-$c_0^+$-sequence. Therefore $C>1/(2b\varrho)$ and hence $\upsilon_n(\mathcal{E})\leqslant 2b\varrho$ for any $n\geqslant N$. Since $\varrho>0$ was arbitrary, we obtain $\lim_n \upsilon_n(\mathcal{E})=0$. 
\end{proof}

\begin{theorem}\label{bco_thm}
Let $a\geqslant 0$, $b\geqslant 1$ and $\mathcal{E}\subset \bigcup_{n=1}^\infty \mathcal{B}^{a,b}_n$ be a~stable set. Then, the following assertions are equivalent:
\begin{enumerate}[{\rm (i)}]
\item $\mathcal{E}$ has block subcotype $q$ for every $1<q<\infty$.
\item $\mathcal{E}$ has block subcotype $q$ for some $1<q<\infty$.
\item $\mathcal{E}$ has block  subcotype.
\item For every $\varrho>0$ there exists $N=N(\varrho)\in \nn$ such that $\mathcal{E}\cap \mathcal{B}^{a,b}_n$ contains no $\varrho$-$c_0$-sequences whenever $n\geq N$. 
\item There do not exist any renormings $|\!\cdot\!|_1$ and $|\!\cdot\!|_2$ of $c_0$ such that the formal identity $\iota\colon (c_0, |\cdot|_1)\to (c_0, |\cdot|_2)$ is block finitely representable on $\eee$. 
\item There does not exist $c>0$ such that $c\cdot\mathrm{Id}_{c_0}$ is block finitely representable on $\mathcal{E}$.
\end{enumerate}
\end{theorem}
\begin{proof} 
(i) $\Rightarrow$ (ii) is trivial.

(ii) $\Rightarrow$ (iii) If $(r, \rho)\in \mathcal{E}\cap \mathcal{B}^{a,b}_n$ and  $(a_i)_{i=1}^n$ is a~scalar sequence satisfying $|a_i\rho(\rme_i)|\geqslant 1$ for each $1\leqslant i \leqslant n$, then 
$$
B_q\gamma_{n,q}(\mathcal{E})\Biggl\|r\Big(\sum_{i=1}^n a_i \ee_i \rme_i\Big)\Biggr\|_{L_1}\geqslant \gamma_{n,q}(\mathcal{E})\Biggl\|r\Big(\sum_{i=1}^n a_i \ee_i \rme_i\Big)\Biggr\|_{L_q} \geqslant \|(a_i \rho(\rme_i))_{i=1}^n\|_{\ell_q^n} \geqslant n^{1/q},
$$
where $B_q$ is a~constant coming from the Kahane--Khintchine inequality. This shows that $\Upsilon_n(\eee)\leq B_q\gamma_{n,q}(\eee)/n^{1/q}$. Therefore, if $\gamma_{q,n}(\mathcal{E})=o(n^{1/q})$, then  $\lim_n \Upsilon_n(\mathcal{E})= 0$ as desired.

\hspace*{-3mm}$\neg$(i) $\Rightarrow$ $\neg$(iii) We modify a~Hinrichs argument \cite{Hinrichs}. Pick $1<q<\infty$ and $\vartheta$ such that 
\begin{equation*}\label{subco_th}
\limsup_{n\to\infty}\frac{\gamma_{q,n}(\mathcal{E})}{n^{1/q}}>\vartheta>0.
\end{equation*}
Fix $n\in \nn$ and pick $N\geq n$, $c>1$ such that $(1-1/c^q)N>a^q n/\vartheta^q$. Take also $M\geqslant N$, $(r, \rho)\in \mathcal{E}\cap \mathcal{B}^{a,b}_M$ and a~scalar sequence $(a_i)_{i=1}^M$ such that 
$$
\|(a_i \rho(\rme_i))_{i=1}^M\|_{\ell_q^M}=M^{1/q}\quad\mbox{and }\,\quad \Biggl\|r\Big(\sum_{i=1}^M a_i \ee_i \rme_i\Big)\Biggr\|_{L_q} \leq 1/\vartheta.
$$
By the $1$-unconditionality of $(\ee_i \rme_i)_{i=1}^M$ in $L_q(\mathbb{K}^M, \rho)$, we have
$$
\max_{1\leqslant i\leqslant M} |a_i \rho(\rme_i)|^q\leqslant a^q \bigl\|(a_i)_{i=1}^M\bigr\|_{\ell_\infty^M} ^q\leqslant a^q \Biggr\|r\Big(\sum_{i=1}^M a_i \ee_i \rme_i\Big)\Biggr\|^q_{L_q} \leqslant \frac{a^q}{\vartheta^q}.
$$
Define $S=\{i\leqslant M\colon |a_i \rho(\rme_i)| \geqslant 1/c\}$ and note that 
$$
M=\sum_{i=1}^M |a_i \rho(\rme_i)|^q \leqslant \frac{a^q}{\vartheta^q}|S|+ \frac{M}{c^q},
$$
whence $|S|>Mn/N\geq n$. Choose a~subset $T$ of $S$ with $|T|=n$; for each $i\in T$ we define $b_i=(a_i \rho(\rme_i))^{-1}\leqslant c$ and note that $\|(b_i a_i \rho(\rme_i))_{i\in T}\|_{\ell_{-\infty}^n}=1$. Then, using $1$-unconditionality once again, we infer that
$$
\Biggl\|r\Big(\sum_{i\in T} b_i a_i \ee_i \rme_i\Big)\Biggr\|_{L_1} \leqslant c \Biggl\|r\Big(\sum_{i=1}^M a_i \ee_i \rme_i\Big)\Biggr\|_{L_q} \leqslant\frac{c}{\vartheta}.
$$
Since this argument works for any $c>0$ and $n\in \nn$, we obtain $\lim\sup_n\Upsilon_n(\mathcal{E})\geqslant \vartheta$. 

The implication (iii) $\Rightarrow$ (iv) is clear, while the converse one follows from an~appropriate modification of Beauzamy's result \cite[Th\'{e}or\`{e}me~2]{Beau}. For reader's convenience, we present an~outline of this argument. 

Assume (iv) holds true and for each $n\in\N$ define
$$
c_n=\inf\Biggl\{\Biggl\|r\Big(\sum_{i=1}^n\e_i a_i\rme_i\Big)\Biggr\|_{L_1}\colon (a_i)_{i=1}^n\subset\mathbb{K}^n,\, \abs{a_i\rho(\rme_i)}\geq 1\,\mbox{ for each }1\leq i\leq n\Biggr\}.
$$
First, we check that $(c_n)_{n=1}^\infty$ is monotone increasing. It is quite easy to see that $(c_n)_{n=1}^\infty$ is monotone increasing ({\it cf.} \cite[Lemme~7]{Beau}). Let $M=\lim_n c_n$ and suppose, towards a~contradiction, that $M<\infty$.

Fix any $\ee\in (0,\frac14)$, $n\in\N$ and define $\delta>0$ by $\delta=(1+\e)/(1-4\e)$. Find $k_0\in\N$ so that $c_k\geq(1-\ee)M$ for each $k\geq k_0$. For any such $k$ let $l=nk$ and $\eta=2^{-l-1}\ee$. Next, pick $m_0\in\N$ so that $c_m\geq(1-\eta)M$ for each $m\geq m_0$. For any such $m$ define $N=m+l$. There exists a~scalar sequence $(a_i)_{i=1}^N$ satisfying $\abs{a_i\rho(\rme_i)}\geq 1$ for each $1\leq i\leq N$ and 
\begin{equation}\label{Beau1}
\Biggl\|r\Big(\sum_{i=1}^N\e_i a_i\rme_i\Big)\Biggr\|_{L_1}\leq (1+\eta)M.
\end{equation}

Define a map $f\colon\{\pm 1\}^\N\to [0,\infty)$ by
$$
f(s)=\int\limits_{\{\pm 1\}^\N}r\Big(\sum_{i=1}^l\e_i(s)a_i\rme_i+\sum_{i=l+1}^N \e_i(t)a_i\rme_i\Big)\,\mathrm{d}t.
$$
We {\it claim} that 
\begin{equation}\label{Beau2}
f(s)\leq (1+\e)M\quad\mbox{for every }s\in\{\pm 1\}^\N.    
\end{equation}
Indeed, observe that \eqref{Beau1} yields $\int f(s)\,\mathrm{d}s\leq (1+\eta)M$. On the other hand, integrating the~triangle inequality of the~form $\frac12(r(x+y)+r(x-y))\geq r(x)$ ({\it cf.} \cite[Lemme~8]{Beau}) we obtain 
$$
f(s)\geq\int\limits_{\{\pm 1\}^\N}r\Big(\sum_{i=l+1}^N\e_i(t)a_i\rme_i\Big)\,\mathrm{d}t\geq (1-\eta)M.
$$
Hence, for any $s\in\{\pm 1\}^\N$ we have
$$
2^{-l}f(s)+(1-2^{-l})(1-\eta)M\leq 2^{-l}f(s)+\int\limits_{\{\pm 1\}^l\setminus\{s\}} f(t)\,\mathrm{d}t\leq (1+\eta)M,
$$
which proves our claim. (Note that the values of $f$ depend only on the~first $l$ coordinates, so we can identify any $s\in\{\pm 1\}^\N$ with its initial sequence of length $l$.)

Using the triangle inequality in a~similar way as above and appealing to \eqref{Beau2}, we obtain that for each $s\in\{\pm 1\}^\N$,
\begin{equation*}
    r\Big(\sum_{i=1}^l \e_i(s)a_i\rme_i\Big)\leq f(s)\leq (1+\e)M.
\end{equation*}
In a very analogous way we can show that for any set $I\subseteq\{1,\ldots,l\}$ and any $s\in\{\pm 1\}^\N$ we have
\begin{equation}\label{Beau_length_l}
r\Big(\sum_{i\in I} \e_i(s)a_i\rme_i\Big)\leq (1+\e)M.
\end{equation}
To this end, we just modify the definition of $f$ by introducing a~function $g$ defined by
$$
g(s)=\int\limits_{\{\pm 1\}^\N}r\Big(\sum_{i\in I}\e_i(s)a_i\rme_i+\!\!\sum_{1\leq i\leq N,\, i\not\in I}\e_i(t)a_i\rme_i\Big)\,\mathrm{d}t,
$$
and we repeat the same argument as above. In particular, we have
\begin{equation}\label{Beau_X}
    r\Big(\sum_{i\in A_j} \e_i(s)a_i\rme_i\Big)\leq (1+\e)M\quad\mbox{for all }s\in\{\pm 1\}^\N,\,\, 1\leq j\leq n,    
\end{equation}
where $A_j=\{(j-1)k+1,\ldots,jk\}$ for $1\leq j\leq n$. Of course, $\abs{A_j}=k\geq k_0$ and hence
\begin{equation}\label{Beau_Y}
    \int\limits_{\{\pm 1\}^\N} r\Big(\sum_{i\in A_j} \e_i(s)a_i\rme_i\Big)\,\mathrm{d}s\geq (1-\e)M\quad\mbox{for every }1\leq j\leq n.   
\end{equation}
Now, inequalities \eqref{Beau_X} and \eqref{Beau_Y}, jointly with the~assumptions we had on the~sequence $(a_i)_{i=1}^N$, imply that for every $1\leq j\leq n$ there exists a~choice of signs $(\delta_i)_{i\in A_j}\subseteq\{\pm 1\}$ such that 
\begin{equation}\label{Beau_C}
r\Big(\sum_{i\in A_j}\delta_i a_i\rme_i\Big)\geq (1-4\e)M\quad\mbox{and }\quad \rho\Big(\sum_{i\in A_j}\delta_i a_i\rme_i\Big)\geq 1
\end{equation}
({\it cf.} \cite[Lemme~9]{Beau}).

For $1\leq j\leq n$ define $y_j=\sum_{i\in A_j}\delta_ia_i\rme_i$ and notice that inequality \eqref{Beau_length_l} yields
$$
r\Big(\sum_{j=1}^n \sigma_jy_j\Big)\leq (1+\e)M\quad\mbox{for every }(\sigma_j)_{j=1}^n\in\{-1,0,1\}^n.
$$
Therefore, for every scalar sequence $(b_j)_{j=1}^n$ we have
\begin{equation*}
    r\Big(\sum_{j=1}^n b_jy_j\Big)\leq (1+\e)M\n{(b_j)_{j=1}^\infty}_{\ell_\infty^n}.
\end{equation*}
Let $z_j=y_j/r(y_j)$ for $1\leq j\leq n$. Then, according to inequality \eqref{Beau_C}, we have $\rho(z_j)\geq ((1-4\e)M)^{-1}$ for each $1\leq j\leq n$, as well as 
\begin{equation*}
    r\Big(\sum_{j=1}^n b_jz_j\Big)\leq \frac{1+\e}{1-4\e}\n{(b_j)_{j=1}^\infty}_{\ell_\infty^n}=(1+\delta)\n{(b_j)_{j=1}^\infty}_{\ell_\infty^n}
\end{equation*}
for every scalar sequence $(b_j)_{j=1}^n$. 

Consequently, the vectors $z_j$ produce a~block of $(r,\rho)$ which forms a~$\varrho$-$c_0$-sequence of length $n$ with $\varrho=\min\{(1+\delta)^{-1},((1-4\e)M)^{-1}\}$. As $n$ was arbitrary, we get a~contradiction with clause (iv). Therefore, $M=\infty$ which means nothing but $\lim_n\Upsilon_n(\eee)=0$, as desired.

The equivalence between assertions (iv)--(vi) can be shown as in the subtype case ({\it cf.} the~proof of Theorem~\ref{btype_thm}).
\end{proof}

\begin{proof}[Proof of Proposition \ref{basic_prop}\,{\rm (ii), (iv), (vi), (viii)}]
(ii) Assume $\eee$ has block basic type $p$ with some $p\in (1,\infty)$. If $(x_i)_{i=1}^\infty$ is block finitely representable on $\eee$ and $(y_i)_{i=1}^\infty$ is any of its normalized block subsequence, then for all $k,n\in\N$ and for the map
$$
\sum_{i=1}^n a_ie_i\xmapsto{\,\,\,\, r\,\,\,\,}\Biggl\|\sum_{i=1}^n a_i y_{k+i}\Biggr\|
$$
we have $(r,r)\in\eee$ which follows from the fact that $\eee$ is block closed. Hence,
$$
\Biggl\|\frac{1}{n}\sum_{i=1}^n y_{k+i}\Biggr\|\leq \alpha_{p,n}(\eee)\cdot\Bigl\|\Bigl(\frac{1}{n}\Bigr)\Bigr\|_{\ell_p^n}=\frac{\alpha_{p,n}(\eee)}{n^{1-1/p}}\xrightarrow[\, n\to\infty\,]{} 0.
$$
Appealing to the fact that a sequence is weakly null if and only if every subsequence admits a~norm null convex block subsequence, we infer that $(x_i)_{i=1}^\infty$ is shrinking. 

The converse is essentially included in Theorem~\ref{bbtype_thm}. First, note that given any operator $A\colon [x_i\colon i\in\N]\to [Ax_i\colon i\in\N]$ block finitely representable on $\eee$, we see that $A$ is an~isometry as $\eee$ is an~identity block. Hence, $\mathrm{Id}_{[x_i\colon i\in\N]}$ is also block finitely representable on $\eee$ and so is $(x_i)_{i=1}^\infty$. By our assumption, every such operator is shrinking and the~implication (v) $\Rightarrow$ (i) of Theorem~\ref{bbtype_thm} yields that $\eee$ has block basic subtype $2$, that is, $\alpha_{2,n}(\eee)=o(n^{1/2})$. Pick $N>1$ so that $\alpha_{2,N}(\eee)<N^{1/2}$, thus for some $\delta\in (0,1)$ we have $\alpha_{2,N}(\eee)\leq N^{\delta-1/2}$. An~appeal to assertion (i) yields that $\eee$ has block basic type $p$, for every $1\leq p<1/\delta$.

(iv) Assume $\eee$ has block basic cotype $q$ with some $q\in (1,\infty)$. If $(x_i)_{i=1}^\infty$ is block finitely representable on $\eee$ and $(z_i)_{i=1}^\infty$ is any its block subsequence satisfying $\delta:=\inf_n\n{z_n}>0$, then for each $n\in\N$ and for the map
$$
\sum_{i=1}^n a_ie_i\xmapsto{\,\,\,\, r\,\,\,\,}\Biggl\|\sum_{i=1}^n a_i \frac{z_i}{\n{z_i}}\Biggr\|
$$
we have $(r,r)\in\eee$. Hence,
$$
\Biggl\|\sum_{i=1}^n z_i\Biggr\|=\Biggl\|\sum_{i=1}^n \n{z_i}\frac{z_i}{\n{z_i}}\Biggr\|\geq \beta_{q,n}(\eee)^{-1}\n{(\n{z_i})_{i=1}^n}_{\ell_q^n}\geq \delta\beta_{q,n}(\eee)^{-1}n^{1/q}\xrightarrow[\, n\to\infty\,]{} \infty.
$$

For the converse, we argue as in the proof of assertion (ii). By Theorem~\ref{bbco_thm}, $\eee$ has block basic subcotype $2$, that is, $\beta_{2,n}(\eee)=o(n^{1/2})$ and hence we may pick $q\in (1,\infty)$ and $N>1$ so that $\beta_{2,N}(\eee)<N^{1/2-1/q}$. By assertion (iii), $\eee$ has block basic cotype $s$ for every $s>q$.

(vi) This is basically the content of Theorem~\ref{btype_thm} by which we infer that the~canonical basis of $\ell_1$ is block finitely representable on $\eee$ if and only if $\eee$ has no nontrivial block subtype (note that here $\eee$ is an~identity block so the constant $c$ from Theorem~\ref{btype_thm} must be $1$). Arguing as in the~proof of (ii) we observe that this is in turn equivalent to $\eee$ having no nontrivial block type.

(viii) By Theorem~\ref{bco_thm}, we infer that the canonical basis of $c_0$ is block finitely representable on $\eee$ if and only if $\eee$ has no nontrivial block subcotype. Arguing as in the proof of (iv) we infer that this is in turn equivalent to $\eee$ having no nontrivial block cotype.
\end{proof}

\section{Operator ideals and seminorms}
Let $\mathsf{Ban}$ denote the class of all Banach spaces over a~fixed field $\mathbb{K}\in\{\R,\C\}$. We denote by $\mathscr{L}$ the class of all operators (always assumed to be linear and bounded) between all Banach spaces and for $X,Y\in \mathsf{Ban}$ we let $\mathscr{L}(X,Y)$ denote the set of all operators from $X$ into $Y$. Generally, for any $\mathscr{I}\subset \mathscr{L}$ and $X,Y\in \mathsf{Ban}$ we set $\mathscr{I}(X,Y)=\mathscr{I}\cap \mathscr{L}(X,Y)$. We recall that a~class $\mathscr{I}$ is called a~({\it two-sided\,}) {\it ideal} if:
\begin{enumerate}[(i)]
\item $\mathrm{Id}_\mathbb{K}\in \mathscr{I}$,
\item for all $X,Y\in \mathsf{Ban}$, $\mathscr{I}(X,Y)$ is a~vector subspace of $\mathscr{L}(X,Y)$, 
\item for any $W,X,Y,Z\in \mathsf{Ban}$, $C\in \mathscr{L}(W,X)$, $B\in \mathscr{I}(X,Y)$ and $A\in \mathscr{L}(Y,Z)$ we have $ABC\in \mathscr{I}$.
\end{enumerate}
An ideal $\mathscr{I}$ is said to be \emph{closed} provided that for any $X,Y\in \mathsf{Ban}$, $\mathscr{I}(X,Y)$ is closed in $\mathscr{L}(X,Y)$ with its norm topology. 

\begin{definition}
{\bf (a)} If $\mathscr{I}$ is an ideal and $\lambda$ assigns to each member of $\mathscr{I}$ a nonnegative number, then we say that $\lambda$ is an~\emph{ideal norm} provided that:
\begin{enumerate}[(i)]
\item for all $X,Y\in \mathsf{Ban}$, $\lambda$ is a~norm on $\mathscr{L}(X,Y)$, 
\item for any $W,X,Y,Z\in \mathsf{Ban}$, $C\in \mathscr{L}(W,X)$, $B\in \mathscr{I}(X,Y)$ and $A\in \mathscr{L}(Y,Z)$ we have $\lambda(ABC)\leqslant \|A\|\lambda(B)\|C\|$, 
\item for any $X,Y\in \mathsf{Ban}$, $x^\ast\in X^\ast$ and $y\in Y$ we have $\lambda(x^\ast\otimes y)=\|x^\ast\|\|y\|$. (Note that all finite-rank operators belong to $\mathscr{I}$). 
\end{enumerate}
\noindent
{\bf (b)} If $\kappa$ assigns to each member of $\mathscr{L}$ a~nonnegative number, then we say $\kappa$ is an~\emph{ideal seminorm} provided that:
\begin{enumerate}[(i)]
\item for all $X,Y\in \mathsf{Ban}$, $\kappa$ is a~seminorm on $\mathscr{L}(X,Y)$, 
\item for any $W,X,Y,Z\in \mathsf{Ban}$, $C\in \mathscr{L}(W,X)$, $B\in \mathscr{I}(X,Y)$ and $A\in \mathscr{L}(Y,Z)$ we have $\kappa(ABC)\leqslant \|A\|\kappa(B)\|C\|$. 
\end{enumerate}
\noindent{\bf (c)} If $\mathscr{I}$ is an ideal and $\lambda$ is an~ideal norm on $\mathscr{I}$, we say $(\mathscr{I}, \lambda)$ is a~\emph{Banach ideal} provided that for all $X,Y\in \mathsf{Ban}$, $(\mathscr{I}(X,Y), \lambda)$ is a~Banach space. 
\end{definition}

We will use the following fact, which seems to be standard. Because we are unaware of a~proof in the literature, we provide one.  
\begin{proposition}\label{ideal prop1}
For every $n\in\N$, let $\lambda_n$ be an~ideal seminorm on $\mathscr{L}$. Suppose that each $\lambda_n$ vanishes on all compact operators and for each $n\in\N$ there exists a~constant $c_n>0$ such that $\lambda_n(A)\leqslant c_n \|A\|$ for every $A\in\mathscr{L}$. Define
$$
\lambda_\infty(A)=\sup_n \lambda_n(A),\,\,\, \mathscr{I}= \{A\in\mathscr{L}\colon \lambda_\infty(A)<\infty\}\,\,\mbox{ and }\,\,\, \lambda(A)=\|A\|+\lambda_\infty(A).
$$
Then, $(\mathscr{I}, \lambda)$ is a~Banach ideal. Furthermore, if the sequence $(c_n)_{n=1}^\infty$ can be taken to be bounded, then the~class
$$
\mathscr{I}_0=\bigl\{A\in\mathscr{L}\colon \lim_n\lambda_n(A)=0\bigr\}
$$
is a~closed ideal. 
\end{proposition}
\begin{proof} 
Since $\lambda_n$ vanishes on compact operators, $\mathscr{I}$ contains all finite-rank operators and for all $X,Y\in \mathsf{Ban}$, $x^\ast\in X^\ast$ and $y\in Y$ we have $\lambda_n(x^\ast\otimes y)= 0$, hence 
$$
\lambda(x^\ast\otimes y)=\|x^\ast\otimes y\|=\|x^\ast\|\|y\|.
$$
The properties of ideal seminorms yield each of the remaining properties needed to know that $\mathscr{I}$ is an~ideal and $\lambda$ is an~ideal norm. We need to show completeness.

Fix $X,Y\in \mathsf{Ban}$ and a~$\lambda$-Cauchy sequence $(A_k)_{k=1}^\infty$ in $\mathscr{I}(X,Y)$. Then $(A_k)_{k=1}^\infty$ is also norm Cauchy, and so it converges in norm to some $A\in \mathscr{L}(X,Y)$. Furthermore, norm continuity of $\lambda_n$ yields that $\lambda_n(A)\leqslant \limsup_k \lambda_n(A_k)$, whence $\lambda_\infty(A) \leqslant\limsup_k \lambda(A_k)<\infty$. A~similar argument yields that 
$$
\limsup_k \lambda_\infty(A-A_k) \leqslant \limsup_j \limsup_k \lambda_\infty(A_j-A_k)=0.
$$
From this it follows that $\lim_k\lambda(A-A_k)=0$. 

Now, suppose that $\sup_n c_n=C<\infty$. The properties of ideal seminorms and the fact that each $\lambda_n$ vanishes on compact operators yield that $\mathscr{I}_0$ is an~ideal. Fix $X,Y\in \mathsf{Ban}$ and define $\Phi\colon\mathscr{L}(X,Y)\to \ell_\infty$ by $\Phi(A)=(\lambda_n(A))_{n=1}^\infty$. Then $\Phi$ is $C$-Lipschitz continuous, whence $\Phi^{-1}(c_0)=\mathscr{I}_0(X,Y)$ is closed. 
\end{proof}

\section{Asymptotic structures and the Szlenk index}
Given a set $D$, we denote by $D^{<\omega}$ the set of all finite sequences in $D$, {\it i.e.} $D^{<\omega}=\bigcup_{n=0}^\infty D^n$. For any $n\in\N$ we set $D^{\leq n}=\bigcup_{i=0}^n D^i$. For any $s,t\in D^{<\omega}$, $s\frown t$ stands for the concatenation of $s$ with $t$ and $|t|$ for the~length of $t$. If $t=(u_i)_{i=1}^n$, then we define $t\vert_m=(u_i)_{i=1}^m$ for $1\leq m\leq n$ and $t\vert_0=\varnothing$. If $t\neq\varnothing$, we denote by $t^-$ its immediate predecessor, that is, $t^-=t|_{|t|-1}$. For $s,t\in D^{<\omega}$ we write $s<t$ provided $|s|<|t|$ and $t|_{|s|}=s$.

Recall that for a~directed set $D$, a~subset $D_1$ of $D$ is called {\it cofinal in} $D$ if for every $u\in D$ there exists $v\in D_1$ with $u\leqslant v$, and it is called {\it eventual} if its complement fails to be cofinal in $D$. Of course, $D_1$ is eventual in $D$ if and only if there exists $u\in D$ such that $u\leqslant v\in D$ implies $v\in D_1$.

\begin{definition}
Let $D$ be a directed set. A subset $D_1\subseteq D^1$ is called {\it full} (resp. {\it inevitable}) provided the set $\{u\in D\colon (u)\in D_1\}$ is cofinal (resp. eventual) in $D$. For $n\in\N$, $T\subseteq D^{n+1}$ and $u\in D$ we set
$$
T(u)=\bigl\{(u_i)_{i=1}^n\in D^n\colon (u,u_1,\ldots,u_n)\in T\bigr\}.
$$
Having defined the notions of being full and inevitable for subsets of $D^k$, for $1\leq k\leq n$, we define $T$ to be {\it full} (resp. {\it inevitable}) provided that $\{u\in D\colon T(u)\mbox{ is full in }D^n\}$ is cofinal in $D$; resp. $\{u\in D\colon T(u)\mbox{ is inevitable in }D^n\}$ is eventual in $D$.
\end{definition}
\begin{definition}
We say a map $\phi\colon D^{\leqslant n}\to D^{\leqslant n}$ is a \emph{pruning} if:
\begin{enumerate}[(i)]
\item $|\phi(t)|=|t|$ for each $t\in D^{\leq n}$;
\item if $t=(u_i)_{i=1}^k\in D^{\leq n}$ and $\phi(t)=(v_i)_{i=1}^k$, then $u_i\leqslant v_i$ in $D$ for each $1\leqslant i\leqslant k$;
\item if $s,t\in D^{\leq n}$, $s<t$, then $\phi(s)<\phi(t)$. \end{enumerate}
\end{definition}

The following assertions are almost trivial to prove, however, they provide a~useful stabilization result which is a~combinatorial tool we shall use quite often.  
\begin{proposition}\label{tech prop}
Let $D$ be a directed set and $n\in \nn$. 
\begin{enumerate}[{\rm (i)}]
\item $\mathcal{E}\subset D^n$ is full if and only if there exists a~pruning $\phi\colon D^{\leqslant n}\to D^{\leqslant n}$ such that $\phi(D^n)\subseteq \mathcal{E}$. 
\item Given $\mathcal{E}\subset D^n$, either $\mathcal{E}$ is full or $D^n\setminus \mathcal{E}$ is inevitable. 
\item If $\mathcal{E}_1, \ldots, \mathcal{E}_k\subseteq D^n$ and $\bigcup_{i=1}^k \mathcal{E}_i$ is full, then there exists $1\leqslant j\leqslant k$ such that $\mathcal{E}_j$ is full. In particular, any finite intersection of inevitable sets is inevitable.
\end{enumerate}
\end{proposition}

\begin{corollary}\label{dpt}
Let $D$ be a directed set and $n\in \nn$. If $D^n=\bigcup_{i=1}^k \mathcal{E}_i$, then there exist $1\leqslant j\leqslant k$ and a~pruning $\phi\colon D^{\leq n}\to D^{\leq n}$ such that $\phi(D^n)\subseteq \mathcal{E}_j$.
\end{corollary}

Now, we are going to define asymptotic structures of operators on Banach spaces which will allow us to introduce aforementioned asymptotic versions of the Rademacher and martingale type/cotype. The study of asymptotic structures of Banach spaces themselves was initiated in \cite{MTJ} and \cite{MMT-J}, where they were defined in terms of games and winning strategies. Our approach is to use trees on Banach spaces mimicking some ideas from \cite{OS_trees}.

By a tree on a Banach space $X$ we mean a~map with a~domain of the form $D^{\leq n}\setminus\{\varnothing\}$, where $D$ is a~directed set, $n\in\N$, and with values in $X$. Every such map will be typically written as $(x_t)_{t\in D^{\leq n}\setminus\{\varnothing\}}\subset X$. We say it is {\it normalized} if $\n{x_t}=1$ for each $t\in D^{\leq n}\setminus\{\varnothing\}$. We say it is {\it weakly null} provided that for every $t\in D^{\leqslant n-1}$, the net $(x_{t\smallfrown (u)})_{u\in D}$ is weakly null. We say $(x^\ast_t)_{t\in D^{\leqslant n}\setminus\{\varnothing\}}\subset X^\ast$ is \emph{weak}$^\ast$-\emph{null} provided that for every $t\in D^{\leqslant n-1}$ the net $(x_{t\smallfrown (u)}^\ast)_{u\in D}$ is weak$^\ast$-null. We say it is \emph{weak}$^\ast$-\emph{closed} if for every $t\in D^{\leqslant n-1}$ we have
$$
x^\ast_t\in \overline{\bigl\{x^\ast_{t\smallfrown (u)}\colon u\in D\bigr\}}^{w\ast}.
$$   
\begin{rem}\upshape 
If $(x_t)_{t\in D^{\leqslant n}\setminus \{\varnothing\}}\subset X$ is weakly null, $\phi\colon D^{\leqslant n}\to D^{\leqslant n}$ is a~pruning and $u_t=x_{\phi(t)}$ for each $\varnothing \neq t\in D^{\leqslant n}$, then $(u_t)_{t\in D^{\leqslant n}\setminus \{\varnothing\}}$ is weakly null. A~similar statement holds for weak$^\ast$-null trees.   
\end{rem}

\begin{definition}
{\bf (a)} Let $A\colon X\to Y$ be an operator between Banach spaces $X$ and $Y$. For $n\in\N$, we define the $n^{\mathrm{th}}$ {\it asymptotic structure of} $A$, denoted $\{A\}_n$, as follows: If $\dim X<\infty$, then we set $\{A\}_n=\varnothing$. Otherwise, $\{A\}_n$ is the~set of all $(r, \rho)\in \mathcal{B}^{\n{A}, 1}_n$ such that for any $\delta>0$ there exist a~directed set $D$ and a~normalized, weakly null tree $(x_t)_{t\in D^{\leqslant n}\setminus \{\varnothing\}}\subset S_X$ such that
$$
d_n\bigl((x_{t|_i}, Ax_{t|_i})_{i=1}^n, (r, \rho)\bigr)<\delta\quad\mbox{for every }t\in D^n.
$$ 
\noindent{\bf (b)} We define the $n^{\mathrm{th}}$ {\it weak$^\ast$ asymptotic structure of} $A$, denoted $\{A\}^\ast_n$, as follows: If we have $\dim Y<\infty$, then $\{A\}_n^\ast=\varnothing$. Otherwise, $\{A\}_n^\ast$ is the set of all $(r, \rho)\in \mathcal{B}^{\|A\|, 1}_n$ such that for any $\delta>0$ there exist a~directed set $D$ and a~normalized, weak$^\ast$-null tree $(y^\ast_t)_{t\in D^{\leqslant n}\setminus \{\varnothing\}}\subset S_{Y^\ast}$ such that
$$
d_n\bigl((y^\ast_{t|_i}, A^\ast y^\ast_{t|_i})_{i=1}^n, (r, \rho)\bigr)<\delta\quad\mbox{for every }t\in D^n.
$$ 
\end{definition}
By standard arguments, the collections 
$$
\mathcal{A}(A)\coloneqq\bigcup_{n=1}^\infty \{A\}_n\quad\mbox{ and }\quad \mathcal{A}^\ast(A)\coloneqq\bigcup_{n=1}^\infty \{A\}_n^\ast
$$
are stable. In particular, if $A$ is the~identity operator, we obtain the~definition of asymptotic structures of $X$, {\it i.e.} $\{X\}_n\coloneqq\{\mathrm{Id}_X\}_n$, and then $\mathcal{A}(X)\coloneqq\mathcal{A}(\mathrm{Id}_X)$ is an~identity block.

\begin{definition}\label{AA_def}
{\bf (a)} Let $A\colon X\to Y$ be an operator between Banach spaces $X$ and $Y$. For $1\leq p,q\leq\infty$ and $n\in\N$ we define:
$$
\alpha_{p,n}(A),\,\, \tau_{p,n}(A),\,\, \beta_{q,n}(A),\,\, \gamma_{q,n}(A),\,\, \theta_n(A),\,\, \Theta_n(A),\,\, \upsilon_n(A),\,\, \Upsilon_n(A)
$$
to be: $\alpha_{p,n}(\mathcal{A}(A))$, $\tau_{p,n}(\mathcal{A}(A))$, $\beta_{q,n}(\mathcal{A}(A))$, $\gamma_{q,n}(\mathcal{A}(A))$, $\theta_n(\mathcal{A}(A))$, $\Theta_n(\mathcal{A}(A))$, $\upsilon_n(\mathcal{A}(A))$, $\Upsilon_n(\mathcal{A}(A))$, respectively. Similarly, we define $\alpha_{n,p}^\ast(A)$, $\tau_{p,n}^\ast(A)$ etc. by replacing $\mathcal{A}(A)$ with $\mathcal{A}^\ast(A)$.

\noindent
{\bf (b)} We say that $A$ has {\it asymptotic} ({\it basic}) ({\it sub}){\it type}/{\it cotype} (equal to $p$ for some $1\leq p\leq\infty$) if $\mathcal{A}(A)$ has the~corresponding block (basic) (sub)type/cotype (equal to $p$). Similarly, we define {\it weak}$^\ast$ {\it asymptotic} ({\it basic}) ({\it sub}){\it type}/{\it cotype} of $A^\ast$ by replacing $\mathcal{A}(A)$ with $\mathcal{A}^\ast(A)$.
\end{definition}

\begin{rem}\label{op_nontrivial}\upshape 
We note that there exist operators with asymptotic (basic) subtype/subcotype and without nontrivial corresponding power type. Indeed, suppose that $1=\theta_1>\theta_2>\ldots$, $\theta_n\downarrow 0$, $2>p_1>p_2>\ldots$, $p_n\downarrow 1$, $2<q_1<q_2<\ldots$ and $q_n\uparrow \infty$. Consider the operator 
$$
A\colon \Bigl(\bigoplus_{n=1}^\infty \ell_{p_n}\oplus \ell_{q_n}\Bigr)_{\ell_2}\xrightarrow[]{\phantom{xxx}} \Bigl(\bigoplus_{n=1}^\infty \ell_{p_n}\oplus \ell_{q_n}\Bigr)_{\ell_2}\,\,\,\mbox{given by }\, A\vert_{\ell_{p_n}\oplus \ell_{q_n}}= \theta_n \mathrm{Id}_{\ell_{p_n}\oplus \ell_{q_n}}
$$
for each $n\in\N$. Then, $A$ has asymptotic basic subtype, asymptotic subtype, asymptotic basic subcotype, and asymptotic subcotype, but no nontrivial power type for any of these four properties, provided that $p_n\downarrow 1$, $q_n\uparrow \infty$ sufficiently rapidly and $\theta_n\downarrow 0$ sufficiently slowly. 
\end{rem}

\begin{rem}\upshape If $A$ is an operator and $\varrho>0$ is such that $\mathcal{A}(A)$ admits arbitrarily long $\varrho$-$\ell_1^+$-sequences, then $A$ cannot have asymptotic type $p$ for $p>2$. Indeed, if $(r, \rho)\in \{A\}_n$ is a $\varrho$-$\ell_1^+$-sequence, then by the geometric Hahn-Banach theorem, there exists $x^\ast\in B_{(\mathbb{K}^n, \rho)^\ast}$ such that $\mathrm{Re}\, x^\ast(\rme_i)\geqslant \varrho$ for each $1\leqslant i\leqslant n$. Then, for each $1\leqslant i\leqslant n$ we may pick $b_i \in [\varrho/a, 1]$ such that $\mathrm{Re}\, x^\ast(b_i \rme_i)=\varrho$. Now, for any sign choice $(\ee_i)_{i=1}^n$ we have 
$$
\rho\Bigl(\sum_{i=1}^n \ee_i b_i \rme_i\Bigr)^2 \geqslant \Bigl|x^\ast\sum_{i=1}^n \ee_i \rme_i\Bigr|^2= \varrho^2\,\Bigl|\sum_{i=1}^n \ee_i\Bigr|^2.
$$ 
Averaging over all $(\ee_i)_{i=1}^n\in\{\pm 1\}^n$ yields
$$
\Biggl\|\rho\Big(\sum_{i=1}^n \ee_i b_i \rme_i\Big)\Biggr\|_{L_2} \geqslant  \varrho n^{1/2}
$$  
and now the Kahane--Khintchine theorem implies the claim. Our next remark shows that what we have just observed is essentially sharp. 
\end{rem}

\begin{rem}\upshape Beanland, Freeman and Motakis \cite{BFM} constructed, for all $1\leqslant p<q\leqslant \infty$, an~example of a~reflexive Banach space $\mathfrak{X}_{p,q}$ with an~unconditional basis and such that:
\begin{enumerate}[(i)]
\item every normalized block sequence in $\mathfrak{X}_{p,q}$ is dominated by the $\ell_p$ basis; 
\item every normalized block sequence in $\mathfrak{X}_{p,q}$ dominates the $\ell_q$ (resp. $c_0$ if $q=\infty$) basis;
\item $\ell_p$ and $\ell_q$ (resp. $c_0$ if $q=\infty$) are spreading models of every block subspace of $\mathfrak{X}_{p,q}$. 
\end{enumerate}
These properties imply that $\mathfrak{X}_{p,q}$ has asymptotic basic type and asymptotic type $p$, as well as asymptotic basic cotype and asymptotic cotype $q$. Moreover, $p$ and $q$ are here optimal.
\end{rem}

In order to reformulate our results from Section~3 in terms of operators, and hence obtain appropriate asymptotic analogues of Beauzamy's theorems, we will need the following asymptotic version of representability. 
\begin{definition}
Let $1\leq p\leq\infty$ and $C\geq 1$. We say that $\ell_p$ ($c_0$ in the case where $p=\infty$) is $C$-\emph{crudely asymptotically finitely representable in} $A$, provided that for each $n\in \nn$ there exists $(r, \rho)\in \{A\}_n$ such that for every scalar sequence $(a_i)_{i=1}^n$ we have
$$
\frac{1}{C}\, r\Bigl(\sum_{i=1}^n a_i \rme_i\Bigr)\leqslant\n{(a_i)_{i=1}^n}_{\ell_p^n}\leq C\rho\Bigl(\sum_{i=1}^n a_i \rme_i\Bigr).
$$
We say $\ell_p$ is \emph{crudely asymptotically finitely representable} in $A$ (resp. \emph{asymptotically finitely representable} in $A$) if it is $C$-crudely asymptotically finitely representable in $A$  for some $C\geq 1$ (resp. for every $C\geq 1$). Similarly, we define the corresponding weak$^\ast$ versions of these notions by replacing $\{A\}_n$ with $\{A\}^\ast_n$. 
\end{definition}

The following result is an~immediate consequence of our results on block structures (Theorems~\ref{bbtype_thm}--\ref{bco_thm}).  
\begin{theorem}\label{main1}
Let $A\colon X\to Y$ be an operator between Banach spaces $X$ and $Y$. 
\begin{enumerate}[{\rm (i)}]
\item $A$ has asymptotic basic subtype if and only if for every $\varrho>0$ there exists $N=N(\varrho)\in\N$ such that $\{A\}_n$ admits no $\varrho$-$\ell_1^+$-sequences whenever $n\geq N$.

\item $A$ has asymptotic basic subcotype if and only if for every $\varrho>0$ there is $N=N(\varrho)\in\N$ such that $\{A\}_n$ admits no $\varrho$-$c_0^+$-sequences whenever $n\geq N$.

\item $A$ has asymptotic subtype if and only if the identity on $\ell_1$ fails to be crudely asymptotically finitely representable in $A$. 

\item $A$ has asymptotic subcotype if and only if the identity on $c_0$ fails to be crudely asymptotically finitely representable in $A$. \end{enumerate}
\end{theorem}
\noindent
Considering the case where $A$ is the~identity operator, and using the~`automatic-type' phenomenon which holds true for identity blocks (Corollary~\ref{aut_type}), we obtain the~following result.
\begin{theorem}\label{main2}
Let $X$ be a Banach space. 
\begin{enumerate}[{\rm (i)}]
\item $X$ has asymptotic basic subtype if and only if for every $\varrho>0$ there exists $N=N(\varrho)\in\N$ such that $\{X\}_n$ admits no $\varrho$-$\ell_1^+$-sequences whenever $n\geq N$, and if and only if $X$ has asymptotic basic type $p$ for some $1<p<\infty$. 

\item $X$ has asymptotic basic subcotype if and only if for every $\varrho>0$ there exists $N=N(\varrho)\in\N$ such that $\{X\}_n$ admits no $\varrho$-$c_0+$-sequences whenever $n\geq N$, and if and only if $X$ has asymptotic basic cotype $q$ for some $1<q<\infty$.

\item $X$ has asymptotic subtype if and only if $\ell_1$ fails to be crudely asymptotically finitely representable in $X$, if and only if $X$ has asymptotic type $p$ for some $1<p<\infty$.  

\item $X$ has asymptotic subcotype if and only if $c_0$ fails to be crudely asymptotically finitely representable in $X$, if and only if $X$ has asymptotic cotype $q$ for some $1<q<\infty$. \end{enumerate}
\end{theorem}

Now, let us recall some basic facts on the Szlenk index and the~Szlenk power type. For more detailed information, the reader may consult the~survey \cite{lancien}. Let $X$ be a Banach space and $K\subset X^\ast$ be a~weak$^\ast$ compact set. For $\ee>0$ we define the so-called {\it Szlenk derivation} by
$$
s_\ee(K)=\Bigl\{x^\ast\in K\colon \mathrm{diam}(V\cap K)>\ee\mbox{ for every weak}^\ast\mbox{ open neighborhood of }x^\ast\Bigr\},
$$
and transfinite derivations by
$$
s_\ee^0(K)=K,\,\,\, s^{\xi+1}_\ee(K)= s_\ee(s_\ee^\xi(K))
$$ 
and for $\xi$ being a~limit ordinal, 
$$
s_\ee^\xi(K)= \bigcap_{\zeta<\xi} s_\ee^\zeta(K).
$$
Next, we define the $\ee$-Szlenk index of $K$ as the minimal ordinal $\xi$ (if exists) for which $s^\xi_\ee(K)=\varnothing$, and we denote it by $Sz(K,\ee)$. If such a~$\xi$ does not exist, we agree to write $Sz(K, \ee)=\infty$ (thus, conventionally, every ordinal $\xi$ satisfies $\xi<\infty$). Next, we set $$
Sz(K)=\sup_{\ee>0} Sz(K,\ee).
$$
Finally, for any Banach spaces $X$, $Y$ and an~operator $A\colon X\to Y$ we define their ($\ee$-){\it Szlenk indices} as follows:
$$
Sz(A,\ee)=Sz(A^\ast B_{Y^\ast},\ee),\,\, Sz(A)=Sz(A^\ast B_{Y^\ast});
$$
$$Sz(X,\ee)=Sz(B_{X^\ast},\ee),\,\, Sz(X)=Sz(B_{X^\ast}).
$$

By a standard compactness argument, $\ee$-Szlenk indices cannot be limit ordinals. Therefore, if $Sz(A)\leqslant \omega$, then for each $\ee>0$ the~value of $Sz(A, \ee)$ must be finite, that is, a~natural number. Thus, we can define the~\emph{Szlenk power type} of $A$ by the formula
$$
\mathsf{p}(A)=\limsup_{\ee\to 0+} \frac{\log Sz(A, \ee)}{\abs{\log(\ee)}}.
$$
In general, this limit need not be finite. However, in the case where $X$ is a~Banach space satisfying $Sz(X)\leqslant\omega$, it is known that the function $(0,1)\ni\ee\xmapsto{\phantom{xx}} Sz(X,\ee)$ is submultiplicative (see \cite[Prop.~4]{lancien}) and hence the~limit defining $\mathsf{p}(X)\coloneqq \mathsf{p}(\mathrm{Id}_X)$ is finite. In fact, we have $1\leq \mathsf{p}(X)<\infty$.

We also need to recall some interplays between the Szlenk index and asymptotic geometry. For an~operator $A\colon X\to Y$ and any $\sigma>0$ we define
$$
\varrho(\sigma, A)=\sup\Bigl\{\limsup_\nu\|y+Ax_\nu\|-1\colon y\in B_Y,\, (x_\nu)\subset \sigma B_X,\, x_\nu\xrightarrow[]{\,\,w\,\,}0\Bigr\}.
$$   
We call $A$ {\it asymptotically uniformly smooth} (in short, AUS) if $\lim_{\sigma\to 0^+}\varrho(\sigma, A)/\sigma=0$. We say $A$ is {\it asymptotically uniformly smoothable} (in short, AUS-{\it able}) if there exists an~equivalent norm $|\!\cdot\!|$ on $Y$ such that $A\colon X\to (Y, |\!\cdot\!|)$ is AUS. For $1<p<\infty$, we say $A$ is $p$-\emph{asymptotically uniformly smoothable} (in short, $p$-AUS) if $\sup_{\sigma>0} \varrho(\sigma, A)/\sigma^p<\infty$.  We say $A$ is $p$-\emph{asymptotically uniformly smoothable} (in short, $p$-AUS-{\it able}) if there exists an~equivalent norm $|\!\cdot\!|$ on $Y$ such that $A\colon X\to (Y, |\!\cdot\!|)$ is $p$-AUS. 

We recall the following renorming theorems associated with these concepts. The spatial version was shown by Godefroy, Kalton and Lancien \cite{GKL}, the operator one by the~first-named author \cite{C} (who also considered some `higher order' versions). 
\begin{theorem}[\cite{GKL}, \cite{C}]\label{causey_glk}
Let $A\colon X\to Y$ be an operator between Banach spaces $X$ and $Y$. 
\begin{enumerate}[{\rm (i)}]
\item  $A$ is {\rm AUS}-able if and only if $Sz(A)\leqslant \omega$. \item If $Sz(A)\leqslant \omega$, then either $A$ is compact or $Sz(A)=\omega$ and 
$$
\mathsf{p}(A)^\prime= \sup\bigl\{p\in (1,\infty)\colon A\text{ is }p\mbox{-{\rm AUS}-able}\bigr\}.
$$
\end{enumerate} 
\end{theorem}

\begin{rem}\label{Szlenk remark}\upshape
It is known that $Sz(X)<\infty$ if and only if $X$ is an Asplund space (see \cite[Chapter I.5]{DGZ} for the real case and \cite[Theorem~1.4]{BA} for deducing the complex case from the real one) which for $X$ separable is equivalent to $X^\ast$ being separable. As it was shown in \cite{Alt}, for any operator $A$, the condition $Sz(A)\leqslant \omega$ is equivalent to the~fact that for some (and then, necessarily, for all) $\varrho\in (0,1)$ there exists $N=N(\varrho)\in \nn$ such that $\{A\}_n$ admits no $\varrho$-$\ell_1^+$-sequence provided $n\geq N$ ({\it cf.} \cite[Theorem~2.2]{Alt}).

These facts together allow us to demonstrate essentially the strongest possible separation between asymptotic basic type and asymptotic type. Indeed, by \cite{PX}, there exists a~nonreflexive Banach space $X$ which has type $p$ and cotype $q$ for all $1<p<2<q<\infty$. By \cite{DJL}, there exists a separable Banach space $Y$ which is finitely representable in $X$ and such that $Y^{**}$ is nonseparable. Classical facts about the behavior of type and cotype under dualization and finite representability yield that $Y$ and $Y^\ast$ have type $p$ and cotype $q$ for all $1<p<2<q<\infty$. Since $Y$ is separable and $Y^{**}$ is not, either $Y$ or $Y^*$ fails to be Asplund. Thus, there exists a~non-Asplund space $Z$ which has type $p$ and cotype $q$ for all $1<p<2<q<\infty$. The space $Z$ being non-Asplund has the property that for each $n\in\N$, $\{Z\}_n$ admits a~$\frac{1}{2}$-$\ell_1^+$-sequence. This means that $Z$ fails to have asymptotic basic subtype in spite of having asymptotic type $p$ for every $1<p<2$. 
\end{rem}

\section{Ideals of interest and duality}
Now, our goal is to provide a~direct description of the quantities introduced in Definition~\ref{AA_def} in terms of weakly null trees. First, let us note the~following fact which can be easily proved by induction. 
\begin{lemma}\label{silly prop}
Let $D$ be a directed set and $X$ be a~Banach space. For any $m\leqslant n\in \nn$, $1\leqslant s_1<\ldots <s_m\leqslant n$ and any weakly null tree $(x_t)_{t\in D^{\leqslant n}\setminus \{\varnothing\}}\subset X$, there exists a~monotone map $\phi\colon D^{\leqslant m}\setminus \{\varnothing\}\to D^{\leqslant n}\setminus\{\varnothing\}$ such that $|\phi(t)|=s_{|t|}$ for every $t\in D^{\leqslant m}\setminus \{\varnothing\}$ and, with the~notation $x^\prime_t=x_{\phi(t)}$, $(x^\prime_t)_{t\in D^{\leqslant m}\setminus \{\varnothing\}}$ is a~weakly null tree as well. 
\end{lemma}

\begin{lemma}\label{primary type lemma}
Let $A\colon X\to Y$ be an operator between Banach spaces $X$ and $Y$. Let also $n\in\N$, $1\leq p\leq\infty$ and $\alpha>0$. Then, the~following assertions are equivalent:
\begin{enumerate}[{\rm (i)}]
\item $\alpha_{p,n}(A)\leqslant \alpha$.
\item For all choices of: a~directed set $D$, a~weakly null tree $(x_t)_{t\in D^{\leqslant n}\setminus \{\varnothing\}}\subset B_X$, a~scalar sequence $(a_i)_{i=1}^n$ and $\alpha^\prime>\alpha$, the set
$$
\Biggl\{t\in D^n\colon \Biggl\|A\sum_{i=1}^n a_i x_{t|_i}\Biggr\|\leqslant \alpha^\prime\n{(a_i)_{i=1}^n}_{\ell_p^n}\Biggr\}
$$
is inevitable.
\item For all choices of: a~directed set $D$, a~weakly null tree $(x_t)_{t\in D^{\leqslant n}\setminus \{\varnothing\}}\subset B_X$, a~scalar sequence $(a_i)_{i=1}^n$, we have
$$
\inf_{t\in D^n} \Biggl\|A\sum_{i=1}^n a_i x_{t|_i}\Biggr\|\leqslant \alpha\n{(a_i)_{i=1}^n}_{\ell_p^n}.
$$  
\end{enumerate}
Regarding $\tau_{p,n}(A)$, for any $\tau>0$, the following assertions are equivalent: 
\begin{enumerate}[{\rm (i)}]
\setcounter{enumi}{3}
\item $\tau_{p,n}(A)\leqslant \tau$.
\item For all choices of: a~directed set $D$, a~weakly null tree $(x_t)_{t\in D^{\leqslant n}\setminus \{\varnothing\}}\subset B_X$, a~scalar sequence $(a_i)_{i=1}^n$ and $\tau^\prime>\tau$, the set
$$
\Biggl\{t\in D^n\colon \Biggl\|A\sum_{i=1}^n \ee_i a_i x_{t|_i}\Biggr\|_{L_p(Y)}\!\!\leqslant\tau^\prime\n{(a_i)_{i=1}^n}_{\ell_p^n} \Biggr\}
$$
is inevitable.
\item For all choices of: a~directed set $D$, a~weakly null tree $(x_t)_{t\in D^{\leqslant n}\setminus \{\varnothing\}}\subset B_X$, a~scalar sequence $(a_i)_{i=1}^n$, we have
$$
\inf_{t\in D^n} \Biggl\|A\sum_{i=1}^n \ee_i a_i x_{t|_i}\Biggr\|_{L_p(Y)}\!\!\leqslant \tau\n{(a_i)_{i=1}^n}_{\ell_p^n}.
$$  
\end{enumerate}
The analogous statements hold for $\alpha^\ast_{p,n}(A)$ and $\tau^\ast_{p,n}(A)$, where instead of weakly null trees we consider weak$^\ast$-null trees in $B_{Y^\ast}$. 
\end{lemma}

\begin{proof} We shall only prove the equivalence between conditions (i)--(iii), as the proof of the second part is quite similar. First, note that all three conditions are obviously true if $\dim X<\infty$, as then we have $\alpha_{p,n}(A)=0$. So, assume $\dim X=\infty$. 

Since any inevitable set is nonempty, we obviously have (ii) $\Rightarrow$ (iii). If (ii) fails, then for some $(a_i)_{i=1}^n\in B_{\ell_p^n}$ and $(x_t)_{t\in D^{\leqslant n}\setminus\{\varnothing\}}\subset B_X$ as above the set 
$$
\mathcal{E}\coloneqq\Biggl\{t\in D^n\colon \Biggl\|A\sum_{i=1}^n a_i x_{t|_i}\Biggr\|>\alpha^\prime\Biggr\}
$$ 
is full. By Proposition \ref{tech prop}, there exists a~pruning $\phi\colon D^{\leqslant n}\to D^{\leqslant n}$ such that $\phi(D^n)\subset \mathcal{E}$. Then $(x_{\phi(t)})_{t\in D^{\leqslant n}\setminus \{\varnothing\}}$ witnesses the failure of (iii). This shows that (ii) $\xLeftrightarrow[]{\phantom{xxx}}$ (iii). 

Assume that (iii) fails. Fix a directed set $D$, a~weakly null tree $(x_t)_{t\in D^{\leqslant n}}\subset B_X$ and $(a_i)_{i=1}^n \in B_{\ell_p^n}$ such that 
$$
\alpha^{\prime}\coloneqq\inf_{t\in D^n} \Biggl\|A\sum_{i=1}^n a_i x_{t|_i}\Biggr\|>\alpha.
$$ 
Pick $\delta>0$ such that $\alpha+\|A\|n \delta<\alpha^{\prime}$. By applying Proposition~\ref{tech prop}, passing to a~subtree and relabeling, we may assume there exist $b_1, \ldots, b_n\in [0,1]$ such that for each $t\in D^n$ and each $1\leqslant i\leqslant n$ we have $\abs{\|x_{t|_i}\|-b_i}<\delta/2$. Define $S=\{i\leqslant n\colon b_i>\delta/2\}$ and notice that $\n{x_{t|_i}}<b_i+\delta/2\leq\delta$ whenever $t\in D^n$ and $i\not\in S$. Therefore, for each $t\in D^n$ we have
\begin{equation*}
\begin{split}
\Biggl\|A\sum_{i\in S} a_i b_i \frac{x_{t|_i}}{\|x_{t|_i}\|}\Biggr\| & \geq \Biggl\|A\sum_{i=1}^n a_i x_{t|_i}\Biggr\|-\|A\|\sum_{i\not\in S} \|a_ix_{t|_i}\|-\|A\|\sum_{i\in S} \abs{a_i}\!\cdot\!\bigl|b_i-\n{x_{t|_i}}\bigr| \\
& \geqslant \alpha^{\prime}-\n{A}\bigl((n-\abs{S})\delta+\abs{S}\delta/2\bigr)\geq\alpha^{\prime}-\|A\|n\delta>\alpha.  
\end{split}
\end{equation*} 

Now, write $S=\{s_1, \ldots, s_m\}$ with $s_1<\ldots <s_m$. By Lemma~\ref{silly prop}, there is a~monotone map $\phi\colon D^{\leqslant m}\setminus \{\varnothing\}\to D^{\leqslant n}\setminus \{\varnothing\}$ such that $|\phi(t)|=s_{|t|}$ for every $t\in D^{\leqslant m}\setminus \{\varnothing\}$ and, with the~notation $x'_t=x_{\phi(t)}/\n{x_{\phi(t)}}$, the~tree $(x'_t)_{t\in D^{\leqslant m}\setminus \{\varnothing\}}$ is weakly null. We may replace $D$ with $D_1=D\times D_X$, where $D_X$ is a~weak neighborhood basis at $0$ in $X$, directing $D_1$ in the obvious way. Define
$$
x'_t=x'_{(u_i)_{i=1}^k}\quad\mbox{for }\,\, t=(u_i, v_i)_{i=1}^k\in D_1^{\leqslant m}\setminus \{\varnothing\}.
$$
Since $X$ is infinite dimensional, we can extend $(x'_t)_{t\in D_1^{\leqslant m}\setminus\{\varnothing\}}$ to a~weakly null tree $(x_t')_{t\in D_1^{\leqslant n}\setminus\{\varnothing\}}\subset S_X$. Set $c_i=a_{s_i}b_{s_i}$ for $1\leqslant i\leqslant m$ and $c_i=0$ for $m<i\leqslant n$. Then 
$$
\inf_{t\in D^n_1} \Biggl\|A\sum_{i=1}^n c_i x'_{t|_i}\Biggr\|\geqslant \alpha^\prime-\|A\|n\delta >\alpha.
$$
Fix a~sequence $N_1\subset N_2\subset \ldots$ of finite subsets of $B_{\ell_\infty^n}$ such that the~union $\bigcup_{i=1}^\infty N_i$ is dense in $B_{\ell_\infty^n}$. Let $u^0_t=x'_t$ for $t\in D^{\leqslant n}\setminus \{\varnothing\}$. By Corollary~\ref{dpt}, we may recursively choose functions $r_i\colon N_i\to \rr$, $\rho_i\colon N_i\to \rr$, weakly null collections $(u^i_t)_{t\in D_1^{\leqslant n}\setminus \{\varnothing\}}\subset S_X$ and prunings $\phi_i\colon D^{\leqslant n}_1\setminus \{\varnothing\} \to D^{\leqslant n}_1\setminus \{\varnothing\}$ such that:
\begin{itemize}
\setlength{\itemindent}{-5mm}
\item $u^i_t=u^{i-1}_{\phi_i(t)}$ for each $i\in\N$;
\item for all $i\in\N$, $(a_j')_{j=1}^n\in N_i$ and $t\in D^n_1$ we have
$$
\Biggl|r_i\Big(\sum_{j=1}^n a_j'\rme_j\Big)-\Biggl\|\sum_{j=1}^n a_j' u^i_{t|_j}\Biggr\|\Biggr|<1/i
$$ 
and 
$$\Biggl|\rho_i\Big(\sum_{j=1}^n a_j' \rme_j\Big)- \Biggl\|A\sum_{j=1}^n a_j' u^i_{t|_j}\Biggr\|\Biggr|<1/i.
$$ 
\end{itemize} 
Then, the limits $r=\lim_i r_i$ and $\rho=\lim_i \rho_i$ exist on $\bigcup_{i=1}^\infty N_i$ and can be extended to produce a~pair $(r, \rho)\in \{A\}_n$. Obviously, $\rho(\sum_{i=1}^n a_i b_i \rme_i)\geqslant \alpha^\prime-\|A\|n \delta>\alpha$ which yields $\alpha_{p,n}(A) \geqslant \alpha^{\prime}-\|A\|n \delta>\alpha$ and therefore (i) fails. 

Finally, assume (i) fails. Fix $(r, \rho)\in \{A\}_n$ and a~scalar sequence $(a_i)_{i=1}^n\in B_{\ell_p^n}$ such that $\rho(\sum_{i=1}^n a_i \rme_i)>\alpha$. Pick $\delta>0$ such that $\rho(\sum_{i=1}^n a_i \rme_i)-\delta>\alpha$, and a~weakly null tree $(x_t)_{t\in D^{\leqslant n}\setminus\{\varnothing\}}\subset S_X$ such that $d_n((r, \rho), (x_{t|_i}, Ax_{t|_i})_{i=1}^n)<\delta$ for each $t\in D^n$. Then 
$$
\inf_{t\in D^n} \Biggl\|A\sum_{i=1}^n a_i x_{t|_i}\Biggr\|\geqslant \rho\Big(\sum_{i=1}^n a_i \rme_i\Big)-\delta>\alpha,$$
thus condition (iii) fails. 
\end{proof}

\begin{rem}\upshape It is essentially contained in the previous proof that $(\alpha_{p,n}(A))_{n=1}^\infty$ is a~nondecreasing sequence. Indeed, the case $\dim X<\infty$ is evident and otherwise if $m\leqslant n$ and $(x_t)_{t\in D^{\leqslant m}\setminus \{\varnothing\}}$, $(a_i)_{i=1}^m$ witness that $\alpha_{p,m}(A)>\alpha$, then we may replace $D$ with $D_1=D\times D_X$, extend $(x_t)_{t\in D^{\leqslant m}_1\setminus\{\varnothing\}}$ to $(x_t)_{t\in D^{\leqslant n}\setminus \{\varnothing\}}$ and let $a_i=0$ for $m<i\leqslant n$. Then $(x_t)_{t\in D_1^{\leqslant n}\setminus \{\varnothing\}}$ and $(a_i)_{i=1}^n$ witness that $\alpha_{p,n}(A)>\alpha$. Similar argument works for the~sequences $(\tau_{p,n}(A))_{n=1}^\infty$, $(\alpha^\ast_{p,n}(A))_{n=1}^\infty$ and $(\tau^\ast_{p,n}(A))_{n=1}^\infty$. 
\end{rem}

\begin{lemma}\label{primary cotype lemma}
Let $A\colon X\to Y$ be an operator between Banach spaces $X$ and $Y$. Let also $n\in\N$, $1\leq q\leq\infty$ and $\beta>0$. Then, the~following assertions are equivalent:
\begin{enumerate}[{\rm (i)}]
\item $\beta_{q,n}(A)\leqslant \beta$.
\item For all choices of: a~directed set $D$, a~bounded weakly null tree $(x_t)_{t\in D^{\leqslant n}\setminus \{\varnothing\}}\subset X$, a~scalar sequence $(a_i)_{i=1}^n \subset [0,\infty)$ such that $\n{Ax_t}\geq a_{\abs{t}}$ for each $\varnothing\neq t\in D^{\leq n}$ and for any $\beta^\prime>\beta$, we have
$$
\beta^\prime \sup_{t\in D^n} \Biggl\|\sum_{i=1}^n x_{t|_i}\Biggr\|\geq \n{(a_i)_{i=1}^n}_{\ell_q^n}.
$$
\item For all choices of: a~directed set $D$, a~bounded weakly null tree $(x_t)_{t\in D^{\leqslant n}\setminus \{\varnothing\}}\subset X$, a~scalar sequence $(a_i)_{i=1}^n \subset [0,\infty)$ such that $\n{Ax_t}\geq a_{\abs{t}}$ for each $\varnothing\neq t\in D^{\leq n}$ and for any $\beta^\prime>\beta$, the set
$$
\Biggl\{ t\in D^n\colon \beta^\prime \sup_{t\in D^n} \Biggl\|\sum_{i=1}^n x_{t|_i}\Biggr\|\geq \n{(a_i)_{i=1}^n}_{\ell_q^n}\Biggr\}
$$
is inevitable.
\end{enumerate}
Regarding $\gamma_{q,n}(A)$, for any $\gamma>0$, the following assertions are equivalent: 
\begin{enumerate}[{\rm (i)}]
\setcounter{enumi}{3}
\item $\gamma_{q,n}(A)\leqslant \gamma$.
\item For all choices of: a~directed set $D$, a~bounded weakly null tree $(x_t)_{t\in D^{\leqslant n}\setminus \{\varnothing\}}\subset X$, a~scalar sequence $(a_i)_{i=1}^n \subset [0,\infty)$ such that $\n{Ax_t}\geq a_{\abs{t}}$ for each $\varnothing\neq t\in D^{\leq n}$ and for any $\gamma^\prime>\gamma$, we have 
$$
\gamma^\prime\sup_{t\in D^n} \Biggl\|\sum_{i=1}^n \ee_i x_{t\vert_i}\Biggr\|_{L_q(X)}\!\! \geqslant \|(a_i)_{i=1}^n\|_{\ell_q^n}.
$$
\item For all choices of: a~directed set $D$, a~bounded weakly null tree $(x_t)_{t\in D^{\leqslant n}\setminus \{\varnothing\}}\subset X$, a~scalar sequence $(a_i)_{i=1}^n \subset [0,\infty)$ such that $\n{Ax_t}\geq a_{\abs{t}}$ for each $\varnothing\neq t\in D^{\leq n}$ and for any $\gamma^\prime>\gamma$, the set 
$$
\Biggl\{t\in D^n\colon \gamma^\prime \Biggl\|\sum_{i=1}^n \ee_i x_{t\vert_i}\Biggr\|_{L_q(X)} \!\!\geq \|(a_i)_{i=1}^n\|_{\ell_q^n}\Biggr\}.
$$
is inevitable.
\end{enumerate}
The analogous statements hold for $\beta^\ast_{q,n}(A)$ and $\gamma^\ast_{q,n}(A)$, where instead of bounded weakly null trees we consider bounded weak$^\ast$-null trees in $Y^\ast$.
\end{lemma}
\begin{proof} 
We again restrict ourselves to the case where $\dim X=\infty$ and we are going to show only the~equivalence between conditions (i)--(iii). We deduce that (ii) $\xLeftrightarrow[]{\phantom{xxx}}$ (iii) as in Lemma~\ref{primary type lemma}. In what follows, we show that (i) $\xLeftrightarrow[]{\phantom{xxx}}$ (ii).

Suppose (i) fails, so that there exist $(r, \rho)\in \{A\}_n$, $\beta'>\beta$ and a~scalar sequence $(a_i)_{i=1}^n$ such that $$
\beta r\Big(\sum_{i=1}^n a_i \rme_i\Big)< \|(a_i\rho(\rme_i))_{i=1}^n\|_{\ell_q^n}=1.
$$
Let $R=\|(a_i)_{i=1}^n\|_{\ell_\infty^n}$ and pick $\delta>0$ such that $\beta r(\sum_{i=1}^n a_i e_i)+ (n+\beta)R\delta<1$. Fix a~directed set $D$ and a~weakly null tree $(x_t)_{t\in D^{\leqslant n}\setminus\{\varnothing\}}\subset S_X$ such that $$
d_n \big((r, \rho), (x_{t|_i}, Ax_{t|_i})_{i=1}^n\big)<\delta\quad\mbox{for each }t\in D^n.
$$
Note that 
$$
\beta\sup_{t\in D^n}\Biggl\|\sum_{i=1}^n a_ix_{t|_i}\Biggr\|\leqslant \beta r\Bigl(\sum_{i=1}^n a_i x_{t|_i}\Bigr)+\beta\delta R
$$ 
and 
$$
\bigl\|a_i Ax_{t|_i}\bigr\| \geqslant \max\bigl\{|a_i|\rho(\rme_i)-\delta R,0\bigr\}=:b_i.
$$   
Since $\|(b_i)_{i=1}^n \|_{\ell_q^n} \geqslant 1-n \delta R$, we deduce that 
$$
\sup_{t\in D^n} \beta\Biggl\|\sum_{i=1}^n a_i x_{t|_i}\Biggr\| \leqslant \beta r\Big(\sum_{i=1}^n a_i \rme_i\Big)+\beta \delta R < 1- n\delta R\leqslant \|(b_i)_{i=1}^n\|_{\ell_q^n},
$$
and thus (ii) fails. 

Now, if (ii) fails, then there exist a~bounded, weakly null tree $(x_t)_{t\in D^{\leqslant n}\setminus \{\varnothing\}}$ and a~scalar sequence $(a_i)_{i=1}^n\subset [0,\infty)$ such that $\|Ax_t\|\geqslant a_{|t|}$ for each $\varnothing\neq t\in D^n$ and 
$$
\beta' \sup_{t\in D^n} \Biggl\|\sum_{i=1}^n x_{t|_i}\Biggr\| < \|(a_i)_{i=1}^n \|_{\ell_q^n}.
$$
We may argue as in the~proof of Lemma~\ref{primary type lemma} to assume there exist scalars $b_1, \ldots, b_n\in [0,1]$ such that for each $\varnothing\neq t\in D^n$, $\|x_t\|\approx b_{|t|}$. Then, ignoring those $i$ for which $b_i$ is small, passing to a~collection $(x_t'/\n{x_t'})_{t\in D^{\leqslant m}\setminus \{\varnothing\}}$, filling out to a~weakly null tree $(x_t)_{t\in D_1^{\leqslant n}\setminus \{\varnothing\}}\subset S_X$, we arrive at scalars $(c_i)_{i=1}^n$ and $\beta<\beta''<\beta'$ satisfying 
$$
\beta''\sup_{t\in D^n}\Biggl\|\sum_{i=1}^n c_i x_{t|_i}\Biggr\| \leqslant \bigl\|(c_i \|Ax_{t|_i}\|)_{i=1}^n\bigr\|_{\ell_q^n}.
$$
We then produce $(r, \rho)\in \{A\}_n$ as in Lemma~\ref{primary type lemma} to satisfy 
$$
\beta''r\Big(\sum_{i=1}^n c_i\rme_i\Big)\leqslant \bigl\|(c_i\rho(\rme_i))_{i=1}^n\bigr\|_{\ell_q^n},
$$
showing that condition (i) fails.
\end{proof}

\begin{theorem}\label{ideal_theorem}
For any $1\leqslant p\leqslant \infty$, the class of operators $A$ for which $\sup_n \alpha_{p, n}(A)$ is finite is a~Banach ideal with the ideal norm $\alpha_p(A)=\|A\|+\sup_n \alpha_{p,n}(A)$. Analogous statements hold for $\tau_{p,n}$, $\beta_{q,n}$, $\gamma_{q,n}$, as well as for the~weak$^\ast$ versions of all these quantities.
\end{theorem}
\begin{proof} 
We will apply Proposition~\ref{ideal prop1}. It is clear that $\alpha_{p,n}(A)$ is nonnegative, satisfies $\alpha_{p,n}(cA)=\abs{c}\alpha_{p,n}(A)$ for any scalar $c$, vanishes on the compact operators and, in view of H\"older's inequality, satisfies $\alpha_{p,n}(A)\leq n^{1/p'}\n{A}$. Fix $W,X,Y,Z\in \mathsf{Ban}$, $C\in \mathscr{L}(W,X)$, $B\in \mathscr{L}(X,Y)$ and $A\in \mathscr{L}(Y,Z)$ with $\|C\|=\|A\|=1$.     Fix also a~directed set $D$ and note that for any $\alpha>\alpha_{p,n}(B)$, any weakly null tree $(w_t)_{t\in D^n\setminus\{\varnothing\}}\subset B_W$ and $(a_i)_{i=1}^n \in B_{\ell_p^n}$ we have
$$
\Biggl\{t\in D^n\colon \Biggl\|ABC\sum_{i=1}^n w_{t|_i}\Biggr\|\geqslant \alpha\Biggr\} \subseteq \Biggl\{t\in D^n\colon \Biggl\|B\sum_{i=1}^n Cw_{t|_i}\Biggr\| \geqslant \alpha\Biggr\}
$$
and since $(Cw_t)_{t\in D^n\setminus\{\varnothing\}}\subset B_X$ is weakly null, the latter set fails to be full. This shows that $\alpha_{p,n}(ABC)\leqslant \alpha_{p,n}(B)$. By homogeneity, $\alpha_{p,n}(ABC)\leqslant \|A\|\alpha_{p,n}(B)\|C\|$. Finally, fix $X,Y\in \mathsf{Ban}$, operators $A,B\colon X\to Y$, $\alpha_A>\alpha_{p,n}(A)$ and $\alpha_B>\alpha_{p,n}(B)$. Then, for any $(a_i)_{i=1}^n\in B_{\ell_p^n}$ and any weakly null $(x_t)_{t\in D^{\leqslant n}\setminus \{\varnothing\}}\subset B_X$ we have
\begin{equation*}
\begin{split}
\Biggl\{t\in D^n\colon \Biggl\|( &A+B)\sum_{i=1}^n x_{t|_i}\Biggr\| \leqslant \alpha_A+\alpha_B\Biggr\}\\
&\supseteq \Biggl\{t\in D^n\colon \Biggl\|A\sum_{i=1}^n x_{t|_i}\Biggr\|\leqslant\alpha_A\Biggr\}\cap \Biggl\{t\in D^n\colon \Biggl\|B\sum_{i=1}^n x_{t|_i}\Biggr\|\leqslant\alpha_B\Biggr\}.
\end{split}
\end{equation*} 
Since the last two sets are inevitable, the intersection of two inevitable sets is inevitable, and supersets of inevitable sets are inevitable, the first set must be inevitable. This shows that $\alpha_{p,n}(A+B)\leqslant \alpha_{p,n}(A)+\alpha_{p,n}(B)$. Thus, Proposition~\ref{ideal prop1} yields the claim regarding $\alpha_{p,n}$. The claims regarding $\tau_{p,n}$, $\alpha^\ast_{p,n}$ and $\tau^\ast_{p,n}$ are similar. 

We now move on to the claims regarding $\beta_{q,n}$. Of course, $\beta_{q,n}(cA)=|c|\beta_{q,n}(A)$ for any operator $A$ and any scalar $c$. To see that $\beta_{q,n}$ satisfies the triangle inequality, we fix Banach spaces $X,Y$, operators $A,B\colon X\to Y$, nonnegative scalars $(c_i)_{i=1}^n$ and a~weakly null collection $(x_t)_{t\in D^{\leqslant n}\setminus \{\varnothing\}}\subset X$ such that for each $\varnothing\neq t\in D^{\leqslant n}$ we have $\|(A+B)x_t\|\geqslant c_{|t|}$. By stabilizing and relabeling, for any $\delta>0$, we may produce scalar sequences $(a_i)_{i=1}^n, (b_i)_{i=1}^n\subset [0,\infty)$ such that 
$$
a_{|t|} \leqslant \|Ax_t\| \leqslant a_{|t|}+\delta\,\,\mbox{ and }\,\, b_{|t|} \leqslant \|Bx_t\|\leqslant b_{|t|}+\delta\quad \mbox{for each }t\in D^{\leqslant n},
$$ 
and 
$$
\delta+ \beta_{q,n}(A)\Biggl\|\sum_{i=1}^n x_{t|_i}\Biggr\| \geqslant \|(a_i)_{i=1}^n\|_{\ell_q^n},
$$
$$
\delta+\beta_{q,n}(B)\Biggl\|\sum_{i=1}^n x_{t|_i}\Biggr\| \geqslant \|(b_i)_{i=1}^n\|_{\ell_q^n}
$$
for every $t\in D^n$. Therefore,
$$
\sup_{t\in D^n} \big(\beta_{q,n}(A)+\beta_{q,n}(B)\big)\Biggl\|\sum_{i=1}^n x_{t|_i}\Biggr\| \geqslant \|(a_i)_{i=1}^n\|_{\ell_q^n}+ \|(b_i)_{i=1}^n\|_{\ell_q^n}-2\delta \geqslant \|(c_i)_{i=1}^n\|_{\ell_q^n}-4\delta.
$$
Since $\delta>0$ was arbitrary, we deduce the triangle inequality for $\beta_{q,n}$. 

Now, let $W,X,Y,Z\in \mathsf{Ban}$ and let $C\colon W\to X$, $B\colon X\to Y$ and $A\colon Y\to Z$ be operators satisfying $\|C\|=\|A\|=1$. Assume also that $(w_t)_{t\in D^{\leqslant n}\setminus \{\varnothing\}}\subset W$ is a~weakly null tree and $(a_i)_{i=1}^n\subset [0,\infty)$ is a~scalar sequence such that $\|ABCw_t\|\geqslant a_{|t|}$ for each $\varnothing\neq t\in D^{\leqslant n}$. Then, $(Cw_t)_{t\in D^{\leqslant n}\setminus\{\varnothing\}}\subset X$ is weakly null and satisfies $\|BCw_t\|\geqslant a_{|t|}$ for each $\varnothing\neq t\in D^{\leqslant n}$. Hence,
$$
\beta_{q,n}(B)\sup_{t\in D^{\leqslant n}} \Biggl\|\sum_{i=1}^n w_{t|_i}\Biggr\| \geqslant \beta_{q,n}(B) \sup_{t\in D^{\leqslant n}} \Biggl\|\sum_{i=1}^n Cw_{t|_i}\Biggr\| \geqslant \|(a_i)_{i=1}^n\|_{\ell_q^n}.
$$   
From this and homogeneity, it follows that $\beta_{q,n}(ABC)\leqslant \|A\|\beta_{q,n}(B)\|C\|$. Finally, since $\beta_{q,n}(A)=0$ for any compact operator $A$, we also have $$
\|x^\ast\otimes y\|+ \sup_n \beta_{q,n}(x^\ast\otimes y) = \|x^\ast\otimes y\|= \|x^\ast\|\|y\|
$$ 
for any $X,Y\in \mathsf{Ban}$, $x^\ast\in X^\ast$ and $y\in Y$.
\end{proof}
\begin{rem}\label{Mr.K}
\upshape 
We note that if $p=\max\{\mathsf{p}(A), 1\}$, then
$$
p^\prime=\sup\bigl\{r\in [1, \infty)\colon A\text{\ has asymptotic basic type\ }r\bigr\}.
$$
(In the case where $A$ is the identity operator it was proved in \cite{DK}). Indeed, if $1<r<p^\prime$, then $A$ is $r$-AUS-able, whence there exists a constant $C$ such that every infinite, weakly null tree in $B_X$ has a~branch whose images under $A$ are $C$-dominated by the $\ell_r$ basis \cite{AUF}. This gives that $\sup_n \alpha_{r,n}(A)\leqslant C$ and $A$ has asymptotic basic type $r$. Conversely, by \cite{Alt}, there exists a~universal constant $c<6$ such that if $Sz(A^*B_{Y^*}, c\ee)>n$, then there exists a~weakly null tree $(x_t)_{t\in D^{\leqslant n}}\subset B_X$ such that for every $t\in D^n$ and every convex combination $x$ of $(x_s\colon \varnothing<s\leqslant t)$ we have $\|Ax\|\geqslant \ee$.  Then 
$$
\ee\leqslant \alpha_{p',n}(A)/n^{1/p}
$$
and hence if $\sup_n \alpha_{p',n}(A)=C<\infty$, we have $Sz(A, c\ee) \leqslant C^p/\ee^p+1$ and $\mathsf{p}(A)\leqslant p$.
\end{rem}

\begin{rem}\upshape 
We remark that the quantity $\alpha_{p,n}(A)$ gives information about uniform domination of members of $\{A\}_n$ by the $\ell_p^n$ basis. The set of $p$ such that $\sup_n \alpha_{p,n}(A)<\infty$ completely determines the supremum of those $p$ for which $A$ is $p$-AUS-able. However, it does not determine whether that supremum is attained. In \cite{AUF}, it was shown that this is determined by the domination of branches of infinite weakly null trees (that is, weakly null trees indexed by $D^{<\omega}$ rather than $D^{\leqslant n}$), and that the class $\mathfrak{T}_p$ of $p$-AUS-able operators can be made into a~Banach ideal with a~natural norm. One can see that for any $1<p\leqslant \infty$, the condition $\sup_n \alpha_{p,n}(A)<\infty$ does not guarantee that $A$ is $p$-AUS-able.  Indeed, consider the  $q$-convexification $T^{(q)}$ of the Figiel--Johnson Tsirelson's space $T$ (see \cite{FJ}) and let $S$ be its dual. Then, $S$ is asymptotic $\ell_p$, whence $\alpha_{p,n}(S)<\infty$. However, $S$ is not $p$-AUS-able. This shows that, if $\mathfrak{A}_p$ is the class of operators $A$ for which $\sup_n \alpha_{p,n}(A)<\infty$, then $\mathfrak{T}_p$ and $\mathfrak{A}_p$ are distinct ideals for every $1<p\leqslant \infty$, and each is a~Banach ideal with an~appropriate norm. Moreover, the argument at the beginning of this remark explains why $\mathfrak{T}_p\subsetneq \mathfrak{A}_p$.
\end{rem}

We may now apply Proposition~\ref{ideal prop1} with $\lambda_n$ being one of the four quantities: $\alpha_{2,n}/n^{1/2}$, $\tau_{2,n}/n^{1/2}$, $\beta_{2,n}/n^{1/2}$, $\gamma_{2,n}/n^{1/2}$, or their weak$^\ast$ counterparts. This is because an~appropriate sequence $(c_n)$ can be taken to be bounded for any one of these eight choices. 
\begin{theorem}\label{various_ideals}
The classes of operators which have {\rm (}weak$^\ast${\rm )} asymptotic {\rm (}basic{\rm )} subtype/sub-\\cotype form a~closed operator ideal. More specifically, we have that:
\begin{enumerate}[{\rm (i)}]
\item The class of operators with Szlenk index not exceeding $\omega$ coincides with the class of operators with asymptotic basic subtype, and is therefore a closed operator ideal. 
\item The class of operators in which $\ell_1$ is not crudely asymptotically finitely representable is a~closed operator ideal. 
\item The class of operators in which $c_0$ is not crudely asymptotically finitely representable is a~closed operator ideal. 
\item The class of operators in which $\ell_1$ is not weak$^\ast$ crudely asymptotically finitely representable is a~closed operator ideal. 
\item The class of operators in which $c_0$ is not weak$^\ast$ crudely asymptotically finitely representable is a~closed operator ideal.
\end{enumerate} 
\end{theorem}

Ostensibly, Theorem~\ref{various_ideals} yields eight different ideals. However, the classes of operators with asymptotic basic subtype and weak$^\ast$ asymptotic basic subcotype coincide. The argument follows the usual technique of dualizing $\ell_1^+$-sequences in asymptotic structures with $c_0^+$-sequences in weak$^\ast$ asymptotic structures of the adjoint. The following result should be compared to the analogous quantified duality between martingale type $p$ of an operator and martingale cotype $p^\prime$ of its adjoint. 

The phenomenon appearing in the following lemma is by now quite standard. Our result concerns $\ell_p^n$ linear combinations of branches of weakly null trees and $\ell_q^n$ linear combinations of weak$^\ast$-null trees, whereas the usual duality (appearing, for example, in \cite{GKL}, \cite{AJO}, and \cite{Alt}) concerns only convex combinations of branches of weakly null trees and `flat' linear combinations of weak$^\ast$-null trees (that is, linear combinations in which all coefficients are equal to $1$).   However, the process is similar, so we omit some of the details. 
\begin{lemma} 
For any $1\leqslant p\leqslant \infty$, $n\in \nn$ and any operator $A\colon X\to Y$ between Banach spaces $X$ and $Y$, we have
$$
\frac{1}{2}\alpha_{p,n}(A)\leqslant \beta^\ast_{p',n}(A)\leqslant 2\alpha_{p,n}(A).
$$
\end{lemma}
\begin{proof}
The assertion is trivial if $A$ is compact, since then we have $\alpha_{p,n}(A)=\beta^\ast_{p',n}(A)=0$. So, in what follows, we assume $A$ is not compact.

Let $\alpha=\alpha_{p,n}(A)$ and suppose that $\beta^\ast_{p',n}(A)>\beta>2 \alpha$. By Lemma~\ref{primary cotype lemma}, there exist a~scalar sequence $(a_i)_{i=1}^n$ with $\n{(a_i)_{i=1}^n}_{\ell_{p'}^n}=1$ and a~normalized weak$^\ast$-null tree $(y^\ast_t)_{t\in D^{\leqslant n}\setminus\{\varnothing\}}$ on $Y^\ast$ such that $\|A^\ast y^\ast_t\|\geqslant |a_{|t|}|$ for each $\varnothing\neq t\in D^n$ and
$$
\beta \sup_{t\in D^n} \Biggl\|\sum_{i=1}^n y^\ast_{t|_i}\Biggr\|<1.
$$
For any $0<\delta,\theta<1$, there exist another directed set $D_1$, a~normalized weakly null collection $(x_t)_{t\in D^n_1\setminus\{\varnothing\}}$ on $X$ and a~monotone, length preserving map $\phi\colon D_1^{\leqslant n}\to D^{\leqslant n}$ such that:
$$
\mathrm{Re}\,y^\ast_{\phi(t)}(Ax_t)\geqslant \frac{\theta|a_{\phi(t)}|}{2}\quad\mbox{for each }\,t\in D_1^{\leqslant n}\setminus\{\varnothing\}
$$
and
$$
|y^\ast_{\phi(s)}(Ax_t)|,\, |y^\ast_{\phi(t)}(Ax_s)|<\delta\quad\mbox{for all }\,s,t\in D_1^{\leqslant n}\setminus\{\varnothing\}\,\mbox{ with }\, s<t. 
$$
Pick a scalar sequence $(b_i)_{i=1}^n\in B_{\ell_p^n}$ such that $\sum_{i=1}^n |a_i|b_i=1$ and $t\in D^n$ such that $\|A\sum_{i=1}^n b_i x_{t|_i}\|<\delta+\alpha_{p,n}(A)$.    Then,
\begin{equation*}
\begin{split}
\beta\Biggl\|\sum_{i=1}^n y^\ast_{\phi(t|_i)}\Biggr\| & \geq \beta(\alpha_{p,n}(A)+\delta)^{-1}\mathrm{Re}\,\Bigl(\sum_{i=1}^n y^\ast_{\phi(t|_i)}\Bigr)\Bigl(A\sum_{i=1}^n b_ix_{t|_i}\Bigr)\\
& \geq \beta(\alpha_{p,n}(A)+\delta)^{-1}\Bigl(\sum_{i=1}^n \theta |a_i|b_i/2- n^2 \delta \Bigr)\\
& =\beta(\alpha_{p,n}(A)+\delta)^{-1}(\theta/2 -n^2\delta).
\end{split}
\end{equation*}
However, by an appropriate choice of $\delta$ and $\theta$ we can guarantee that the last expression is larger than $1$, which gives a~contradiction. 

Now, assume that $\alpha_{p,n}(A)>\alpha>2 \beta^\ast_{p',n}(A)$. Then there exist a~weakly null collection $(x_t)_{t\in D^{\leqslant n}\setminus \{\varnothing\}}\subset S_X$ and a~scalar sequence $(a_i)_{i=1}^n$ with $\n{(a_i)_{i=1}^n}_{\ell_p^n}=1$ such that 
$$
\inf_{t\in D^n} \Biggl\|A\sum_{i=1}^n a_i x_{t|_i}\Biggr\|>\alpha.
$$
For each $t\in D^n$, pick $y^\ast_t\in B_{Y^\ast}$ such that 
$$
\mathrm{Re}\,y^\ast_t\Big(A\sum_{i=1}^n a_i x_{t|_i}\Big)=\Biggl\|A\sum_{i=1}^n a_i x_{t|_i}\Biggr\|.
$$
Fix $\delta>0$ such that $\alpha+3\delta n<\inf_{t\in D^n} \|A\sum_{i=1}^n x_{t|_i}\|$. After stabilizing and relabeling, we may find real scalars $(b_i)_{i=1}^n$ such that 
$$
\big|\mathrm{Re}\,y^\ast_t(Ax_s)- b_{|s|}\big|\leqslant \delta\quad\mbox{for all }\, \varnothing<s\leqslant t\in D^n
$$
and $\sum_{i=1}^n b_i>\alpha+ 3 n \delta$.

Let $D_1=D_X\times D_{Y^*}$, where $D_X$ is a weak neighborhood basis at $0$ in $X$ and $D_{Y^\ast}$ is a~weak$^\ast$ neighborhood basis at $0$ in $Y^\ast$. By weak$^\ast$ compactness, there exists a~collection $(y^\ast_t)_{t\in D^{\leqslant n-1}}\subset B_{Y^\ast}$ such that $(y^\ast_t)_{t\in D^{\leqslant n}}$ is weak$^\ast$ closed. We may then find a~monotone, length preserving map $\phi\colon D_1^{\leqslant n}\setminus \{\varnothing\}\to D^{\leqslant n}\setminus \{\varnothing\}$ such that, with the notation
$$
z_t=x_{\phi(t)},\quad z^\ast_t= y^\ast_{\phi(t)}- y^\ast_{\phi(t^-)},
$$
the collection $(z_t)_{t\in D^{\leqslant n}\setminus \{\varnothing\}}$ is weakly null, while $(z^\ast_t)_{t\in D^{\leqslant n}\setminus \{\varnothing\}}$ is weak$^\ast$-null and, moreover,
$$
\mathrm{Re}\,z^\ast_t(Az_t) \geqslant b_{|t|}-2\delta\quad\mbox{for each }\, t\in D_1^{\leq n}\setminus\{\varnothing\}
$$
and
$$
|z_t^\ast(Az_s)|,\, |z^\ast_s(Az_t)|<\delta\quad\mbox{for all }\, s,t\in D_1^{\leq n}\setminus\{\varnothing\}\,\mbox{ with }\, s<t.
$$
Define $S=\{i\leqslant n\colon b_i-2\delta>0\}$ and let $c_i=b_i-2\delta$ for $i\in S$ and $c_i=0$ otherwise. Since $\|A^\ast z^\ast_t\|\geqslant c_{|t|}$, we obtain
$$
2\beta^\ast_{p',n}(A) \geqslant  2\beta^\ast_{p',n}(A) \sup_{t\in D^n} \Biggl\|\sum_{i=1}^n z^\ast_{t|_i}\Biggr\|\geqslant \|(c_i)_{i=1}^n\|_{\ell_{p'}^n}.
$$
However, 
\begin{equation*}
\begin{split}
\alpha +3\delta n & <  \delta n+\sum_{i\in S}^n a_i b_i \leqslant 3\delta n+\sum_{i\in S} a_i b_i \leqslant 3\delta n+ \|(a_i)_{i=1}^n\|_{\ell_p^n}\|(c_i)_{i=1}^n\|_{\ell_{p'}^n}\\
& \leqslant 3\delta n+ 2 \beta^\ast_{p',n}(A);
\end{split}
\end{equation*}
a contradiction.
\end{proof}

This immediately yields the following corollary giving the exact duality between asymptotic basic type and weak$^\ast$ asymptotic basic cotype. This result parallels the duality between martingale type $p$ and martingale cotype $p^\prime$.  
\begin{corollary}\label{duality_for_operators}
Let $A\colon X\to Y$ be an operator between Banach spaces and let $1\leq p\leq\infty$. Then $A$ has asymptotic basic type $p$ if and only if $A^\ast$ has weak$^\ast$ asymptotic basic cotype $p^\prime$. It has asymptotic basic subtype if and only if $A^\ast$ has weak$^\ast$ asymptotic basic subcotype. 
\end{corollary}

\section{A three-space property}
In this final section, we will prove our main three-space property result (Theorem~\ref{BL_main} below). It yields an~optimal improvement of the Brooker--Lancien theorem from \cite{BL}. They showed that if $Y$ is a~closed subspace of a~Banach space $X$, then $\mathsf{p}(X)\leqslant \mathsf{p}(Y)+\mathsf{p}(X/Y)$. We will show that instead of addition one can take maximum and this is, of course, sharp.

\begin{lemma}\label{cool}
Let $X$ be a Banach space and $Y$ a closed subspace of $X$. Let $D$ be a~directed set, $m\in \nn$ and $(x_t)_{t\in D^{\leqslant m}\setminus\{\varnothing\}}\subset X$ be a~weakly null tree. Assume that we are given functions $f,g\colon \{\pm 1\}^m\to (0, \infty)$ such that 
$$
\Biggl\|\sum_{i=1}^m \ee_i x_{t|_i}\Biggr\|<f(\ee)\quad\mbox{and }\quad \Biggl\|\sum_{i=1}^m \ee_i x_{t|_i}\Biggr\|_{X/Y}\!\!<g(\ee)\quad\mbox{for all }\ee\in\{\pm 1\}^m,\,\, t\in D^m.
$$
Then, for any weak neighborhood $u$ of $0$ in $X$ and any weak neighborhood $v$ of $0$ in $Y$, there exists $t\in D^m$ with the~property that for each $\ee\in \{\pm 1\}^m$ there is a~vector $y_\ee\in Y$ such that: 
\begin{enumerate}[{\rm (i)}]
\item $\sum_{i=1}^m \ee_i x_{t|_i}\in u$,
\item $\|y_\ee\|\leqslant 2 f(\ee)$, 
\item $\|y_\ee-\sum_{i=1}^m \ee_i x_{t|_i}\|< 5 g(\ee)$, 
\item $y_\ee\in v$. 
\end{enumerate}
\end{lemma}
\begin{proof} 
By replacing $u$ and $v$ with a subset, we may assume that both of them are convex and symmetric. Let $D_0$ be the family of all convex, symmetric, weak neighborhoods of $0$ in $X$, directed by reverse inclusion. For every $w\in D_0$, there exists $t_w\in D^m$ such that $x_{t_w|_i}\in \frac{1}{m} w$ for each $1\leq i\leq m$. For any $\ee\in \{\pm 1\}^m$ define
$$
X_w(\ee)=\sum_{i=1}^m \ee_i x_{t_w|_i}\in w.
$$
Observe that the collection $((X_w(\ee))_{\ee\in \{\pm 1\}^m})_{w\in D_0}$ forms a~weakly null net in $\ell_\infty^{2^m}(X)$. By our assumption, for each $w\in D_0$ and each $\ee\in \{\pm 1\}^m$, we may pick $Z_w(\ee)\in Y$ such that $\|X_w(\ee)-Z_w(\ee)\|<g(\ee)$. Define
$$
Y_w(\ee)=\left\{\begin{array}{rl}
\displaystyle{\frac{f(\ee)}{\|Z_w(\ee)\|}Z_w(\ee)} & \mbox{if }\,\|Z_w(\ee)\|>f(\ee)\\
Z_w(\ee) & \mbox{if }\,\n{Z_w(\ee)}\leq f(\ee).
\end{array}
\right.
$$ 
Note that $\|Y_w(\ee)\|\leqslant f(\ee)$ and $\|X_w(\ee)-Y_w(\ee)\|< 2 g(\ee)$. Now, we may pass to subnets $(X_w(\ee))_{w\in D_1}$ and $(Y_w(\ee))_{w\in D_1}$ such that $X_w(\ee)\in u$ and $Y_{w_1}(\ee)-Y_{w_2}(\ee)\in v$ for all $w,w_1,w_2\in D_1$, $\ee\in\{\pm 1\}^m$. Here, we are using the Banach--Alaoglu theorem together with the boundedness of the net $(Y_w(\ee))_{w\in D_0}$ in $\ell_\infty^{2^m}(Y)^{\ast\ast}$. By weak nullity, there exist a finite set $F\subset D_1$ and positive numbers $(a_w)_{w\in F}$ summing to $1$ such that 
$$
\Biggl\|\sum_{w\in F}a_w X_w(\ee)\Biggr\|_{\ell_\infty^{2^m}(X)}<g(\ee).
$$
Now, choose any $w_1\in D_1$ and for any fixed $\ee\in\{\pm 1\}^m$ let 
$$
y_\ee=Y_{w_1}(\ee)-\sum_{w\in F} a_w Y_w(\ee)=\sum_{w\in F} a_w(Y_{w_1}(\ee)-Y_w(\ee))\in v.
$$
We shall show that $t=t_{w_1}$ is the desired element of $D^m$. Assertions (i) and (iv) are immediate from the above construction. To see that (ii) and (iii) are also true, note that
$$
\|y_\ee\| \leqslant \|Y_{w_1}(\ee)\|+\sum_{w\in F}a_w\|Y_w(\ee)\| \leqslant 2f(\ee)
$$
and 
\begin{equation*}
\begin{split}
\|y_\ee-X_{w_1}(\ee)\|\leq &\|Y_{w_1}(\ee)-X_{w_1}(\ee)\|+\! \sum_{w\in F} a_w\|Y_w(\ee)-X_w(\ee)\|\\
& + \Biggl\|\sum_{w\in F}a_w X_w(\ee)\Biggr\|\leqslant 5g(\ee),
\end{split}
\end{equation*}
which completes the proof.
\end{proof}

We also have the following unsigned version of the~previous result, which is even easier. 
\begin{lemma} 
Let $X$ be a Banach space and $Y$ a closed subspace of $X$. Let $D$ be a~directed set, $m\in \nn$ and $(x_t)_{t\in D^{\leqslant m}\setminus\{\varnothing\}}\subset X$ be a~weakly null tree. Assume that there are $a,b>0$ such that 
$$
\Biggl\|\sum_{i=1}^m x_{t|_i}\Biggr\|<a\quad\mbox{and }\quad \Biggl\|\sum_{i=1}^m x_{t|_i}\Biggr\|_{X/Y}\!\!\!<b\quad\mbox{for every }\, t\in D^m.
$$
Then, for any weak neighborhood $u$ of $0$ in $X$ and any weak neighborhood $v$ of $0$ in $Y$, there exist $t\in D^m$ such that:
\begin{enumerate}[{\rm (i)}]
\item $\sum_{i=1}^m x_{t|_i}\in u$,
\item $\|y\|\leqslant 2a$, 
\item $\|y-\sum_{i=1}^m x_{t|_i}\|< 5b$, 
\item $y\in v$. 
\end{enumerate}
\label{2cool}
\end{lemma}

The next lemma is crucial for our three-space result. It corresponds to similar estimates for the Rademacher and martingale types obtained by Enflo, Lindenstrauss and Pisier ({\it cf.} \cite[Thm.~1]{ELP}).
\begin{lemma} 
Let $X$ be a Banach space and $Y$ a closed subspace of $X$.
Then for all $1\leqslant p\leqslant \infty$ and $m,n\in \nn$ we have
$$
\alpha_{p,mn}(X) \leqslant 2\alpha_{p,m}(X)\alpha_{p,n}(Y)+5\alpha_{p,n}(X)\alpha_{p,m}(X/Y)
$$
and 
$$
\tau_{p,mn}(X) \leqslant 2\tau_{p,m}(X)\tau_{p,n}(Y)+5\tau_{p,n}(X)\tau_{p,m}(X/Y).
$$
\label{gen}
\end{lemma}
\begin{proof} 
We prove only the second inequality. The first can be proved along the same lines replacing the Rademacher coefficients $(\ee_i)$ with constant coefficients and using Lemma~\ref{2cool} instead of \ref{cool}. Furthermore, if either $\dim Y<\infty$ or $\dim X/Y<\infty$, the result is clear, so we assume that $\dim Y=\dim X/Y=\infty$. In this case, $\tau_{p,n}(Y),\tau_{p,n}(X/Y)\geqslant 1$. 

Fix a weakly null tree $(x_t)_{t\in D^{\leqslant mn}\setminus \{\varnothing\}}\subset B_X$ and a~scalar sequence $(a_i)_{i=1}^{mn}$. For each $1\leqslant i\leqslant n$ choose any $b_i>\n{(a_j)_{(i-1)m<j\leqslant im}}_{\ell_p^m}$. For any $\delta>0$, by applying a~suitable pruning, we may assume there exist functions $\eta, \mu\colon\{\pm 1\}^{mn}\times \{1, \ldots, n\}\to \rr$ such that for all $t\in D^{mn}$, $\ee=(\ee_j)_{j=1}^{mn}\in \{\pm 1\}^{mn}$ and $1\leqslant i\leqslant n$:
\begin{itemize}
\setlength{\itemindent}{-5mm}
\setlength\itemsep{2pt}
\item the~values of $\eta(\cdot, i)$ and $\mu(\cdot, i)$ depend only on $\ee_{(i-1)m+1},\ldots,\ee_{im}$,
\item $\Big|\eta(\ee,i)-\left\|\sum_{j=(i-1)m+1}^{im} \ee_ja_jx_{t\vert_j}\right\|\Big|<\delta\,$
\item $\Big|\mu(\ee,i)-\left\|\sum_{j=(i-1)m+1}^{im} \ee_ja_jx_{t\vert_j}\right\|_{X/Y}\Big|<\delta$.
\end{itemize}
Furthermore, decreasing $\delta$ if necessary, we may also assume that
$$
\|\eta(\ee, i)\|_{L_p(\{\pm 1\}^{mn})} \leqslant \tau_{p,m}(X)b_i\quad\mbox{ and }\quad \|\mu(\ee, i)\|_{L_p(\{\pm 1\}^{mn})} \leqslant \tau_{p,m}(X/Y) b_i.
$$    

Now, let $D_1=D_X\times D_Y$, where $D_X$ is a~weak neighborhood basis at $0$ in $X$ and $D_Y$ is a~weak neighborhood basis at $0$ in $Y$. For every $\ee\in \{\pm 1\}^{mn}$, we are going to construct weakly null trees $(X_t(\ee))_{t\in D_1^{\leqslant n}\setminus \{\varnothing\}}\subset X$ and $(Y_t(\ee))_{t\in D_1^{\leqslant n}\setminus \{\varnothing\}}\subset Y$, and a~map $\phi\colon D_1^{\leqslant n}\to D^{\leqslant mn}$ such that:
\begin{itemize}
\setlength{\itemindent}{-5mm}
\setlength\itemsep{2pt}
\item $|\phi(t)|=m|t|$ for each $t\in D_1^{\leqslant n}$, \item $Y_t(\ee)$ is a function only of $\ee_{(|t|-1)m+1}, \ldots, \ee_{|t|m}$, 
\item $\|Y_t(\ee)\|\leqslant 2(\eta(\ee, |t|)+\delta)$,
\item $\|X_t(\ee)-Y_t(\ee)\|<5(\mu(\ee, |t|)+\delta)$, 
\item $X_t(\ee)=\sum_{j=(|t|-1)m+1}^{|t|m} \ee_j a_j x_{\phi(t)|j}$.
\end{itemize}
We define $\phi(t)$ by induction on $|t|$ as follows. First, set $\phi(\varnothing)=\varnothing$. If $s\in D_1^{\leq n}\setminus\{\varnothing\}$ and $\phi(s^-)$ has been defined, we apply Lemma~\ref{cool} to the~weakly null tree $(x_{\phi(s^-)\smallfrown t})_{t\in D^{\leqslant m}\setminus\{\varnothing\}}$ and the functions 
$$
f(\ee)=\eta(\ee, |s|),\quad g(\ee)=\mu(\ee, |s|).
$$
This yields appropriate $t\in D^m$ and $Y_s(\ee)\in Y$. We define $\phi(s)=\phi(s^-)\!\frown\! t$ and then it is easy to check that all the~desired conclusions hold. 

After pruning once again, we may additionally assume that for any $t\in D^n_1$ and $\ee\in \{\pm 1\}^{mn}$ we have
$$
\Biggl\| \sum_{i=1}^n \ee_i^\prime Y_{t|_i}(\ee)\Biggr\|_{L_p(\{\pm 1\}^n)}\!\!\! \leqslant 2\big(\tau_{p,n}(Y)+\delta\big) \bigl\|(\eta(\ee, i)+\delta)_{i=1}^n\bigr\|_{\ell_p^n}
$$
and 
$$
\Biggl\|\sum_{i=1}^n \ee_i^\prime \big(X_{t|_i}(\ee)-Y_{t|_i}(\ee)\big)\Biggr\|_{L_p(\{\pm 1\}^n)}\!\!\! \leqslant 5\big(\tau_{p,n}(X)+\delta\big)\bigl\|(\mu(\ee, i)+\delta)_{i=1}^n\bigr\|_{\ell_p^n}.
$$
(Here, the $L_p$-norms are taken with respect to the variable $\ee^\prime\in\{\pm 1\}^n$ and we use estimates on $\n{Y_t(\ee)}$ and $\n{X_t(\ee)-Y_t(\ee)}$ listed above). Then, for every $t\in D_1^n$, the triangle inequality gives
\begin{equation*}
\begin{split}
\Biggl\|\sum_{i=1}^n X_{t|_i}(\ee)\Biggr\|_{L_p(\{\pm 1\}^{mn})}\!\!\! &=\Biggl\|\sum_{i=1}^n \ee_i' X_{t|_i}(\ee)\Biggr\|_{L_p(\{\pm 1\}^n\times \{\pm 1\}^{mn})}\\
&\leq \Biggl\|\sum_{i=1}^n \ee_i' \big(X_{t|_i}(\ee)-Y_{t|_i}(\ee)\big)\Biggr\|_{L_p(\{\pm 1\}^n\times \{\pm 1\}^{mn})}\\
&\quad +\Biggl\|\sum_{i=1}^n \ee_i' Y_{t|_i}\Biggr\|_{L_p(\{\pm 1\}^n\times \{\pm 1\}^{mn})} \\
&\leq 2\tau_{p,n}(Y) \bigl\|(\|\eta(\ee, i)\|_{L_p(\{\pm 1\}^{mn})})_{i=1}^n\bigr\|_{\ell_p^n}\\
&\quad +5\tau_{p,n}(X)\bigl\|(\|\mu(\ee, i)\|_{L_p(\{\pm 1\}^{mn})})_{i=1}^n \bigr\|_{\ell_p^n}+\mathrm{o}(1) \\
&\leq \bigl(2\tau_{p,n}(Y)\tau_{p,m}(X)+5\tau_{p,n}(X)\tau_{p,m}(X/Y)\bigr)\|(b_i)_{i=1}^n\|_{\ell_p^n}\\
&\hspace{78mm}+\mathrm{o}(1),
\end{split}
\end{equation*}
where $\mathrm{o}(1)\to 0$ as $\delta\to 0$. Recall that $\delta>0$ could be arbitrarily small and $b_i$ arbitrarily close to $\|(a_j)_{(i-1)m<j\leqslant im}\|_{\ell_p^m}$. Therefore, by appealing to Lemma~\ref{primary type lemma}, we obtain the desired inequality.
\end{proof}

\begin{lemma}[\cite{ELP}]\label{ELP_lemma}
If $(a_n)_{n=1}^\infty$ is a nondecreasing sequence of positive numbers and $C>0$ is such that $a_{n^2} \leqslant Ca_n$ for $n\in\N$, then there exist numbers $C_1, \lambda>0$ such that 
$$
a_n \leqslant C_1 (\log n)^\lambda\quad\mbox{for each }n\geq 2.
$$
\end{lemma}

Finally, we are in a position to prove our main result. By $\mathsf{t}(X)$ we denote the~supremum over those $1\leq p\leq \infty$ for which $X$ has asymptotic type $p$.
\begin{theorem}\label{BL_main}
Let $X$ be a Banach space and $Y$ a closed subspace of $X$. Then, we have
$$
\mathsf{p}(X)=\max\{\mathsf{p}(Y), \mathsf{p}(X/Y)\}
$$
and
$$\mathsf{t}(X)= \min\{\mathsf{t}(Y), \mathsf{t}(X/Y)\}.
$$
Therefore, for any $1\leqslant p,t\leqslant\infty$, `$\,\mathsf{p}(\,\cdot\,)\leqslant p$' and `$\,\mathsf{t}(\,\cdot\,)\geqslant t$' are three-space properties. In particular, `$\,Sz(\,\cdot\,)\leqslant \omega$' and `$\,\ell_1$ is not asymptotically finitely representable in $\cdot$\,' are three-space properties.
\end{theorem}
\begin{proof} 
We know that $\mathsf{p}(X)\geqslant \max\{\mathsf{p}(Y), \mathsf{p}(X/Y)\}$ (see, {\it e.g.}, \cite[\S 2.4]{hsmvz}), so in the case where one of $\mathsf{p}(Y)$, $\mathsf{p}(X/Y)$ is infinite there is nothing to show. Neither is there if $Y$ or $X/Y$ is finite dimensional. In what follows, we assume that none of these trivial situations holds.

Fix any numbers $1<r<p<\min\{\mathsf{p}(Y)^\prime, \mathsf{p}(X/Y)^\prime\}$. Notice that, according to Remark~\ref{Mr.K}, we have
$$
C\coloneqq\sup_n \max\{\alpha_{p,n}(Y), \alpha_{p,n}(X/Y)\}<\infty.
$$
Hence, Lemma~\ref{gen} implies that $\alpha_{p,n^2}(X) \leqslant 7C \alpha_{p,n}(X)$ for each $n\in\N$. Now, in view of Lemma~\ref{ELP_lemma}, there exist constants $C_1, \lambda>0$ such that $\alpha_{p,n}(X) \leqslant C_1 (\log n)^\lambda$ for each $n\geq 2$. For a~sufficiently large $N\in\N$ we must have $\alpha_{p,N}(X) <N^{\frac{1}{r}-\frac{1}{p}}$, whence Proposition~\ref{basic_prop}(i) implies that $X$ has asymptotic basic type $r$. Appealing once again to Remark~\ref{Mr.K} we infer that $\mathsf{p}(X) \leqslant r^\prime$ and by the choice of $r$ and $p$ we finally obtain $\mathsf{p}(X) \leqslant \max\{\mathsf{p}(Y),\mathsf{p}(X/Y)\}$.  

The same reasoning applies to $\mathsf{t}$. 
\end{proof}



\begin{thebibliography}{HD}
\bibitem{AJO} D. Alspach, R. Judd, E. Odell, \emph{The Szlenk index and local $\ell_1$-indices}, Positivity \textbf{9} (2005), 1--44.

\bibitem{BFM} K. Beanland, D. Freeman, P. Motakis, \emph{The stabilized set of p's in Krivine's theorem can be disconnected}, Adv. Math. \textbf{281} (2015), 553--577.

\bibitem{Beau} B. Beauzamy, \emph{Op\'{e}rateurs de type Rademacher entre espaces de Banach}, S\'{e}m Maurey-Schwartz, Expos\'{e}s 6 et 7, \'{E}cole Polyt\'{e}chnique, Paris (1975/76). 

\bibitem{benyamini} Y. Benyamini, J. Lindenstrauss, \emph{Geometric nonlinear functional analysis}, Amer. Math. Soc., Providence, Rhode Island~2000.

\bibitem{BA} P.A.H. Brooker, \emph{Asplund operators and the Szlenk index}, Operator Theory \textbf{68} (2012),  405--442.

\bibitem{BL} P.A.H. Brooker, G. Lancien, \emph{Three-space properties and asymptotic structure of Banach spaces}, J. Math. Anal. Appl. \textbf{398} (2013), 867--871.

\bibitem{Alt} R.M. Causey, \emph{An alternate description of the Szlenk index with applications}, Illinois J. Math. \textbf{59} (2015), 359--390.

\bibitem{C} R.M. Causey, \emph{Power type $\xi$-asymptotically uniformly smooth norms}, to appear in Trans. Amer. Math. Soc., DOI: {\tt https://doi.org/10.1090/tran/7336}

\bibitem{AUF} R.M. Causey, \emph{Power type asymptotically uniformly smooth and asymptotically uniformly flat norms}, {\tt arXiv:1705.05484} 

\bibitem{DJL} W.J. Davis, W.B. Johnson, J. Lindenstrauss, \emph{The $\ell_1^n$ problem and degrees of nonreflexivity}, Studia Math. \textbf{55} (1976), 123--139. 

\bibitem{DGZ} R. Deville, G. Godefroy, V. Zizler, \emph{Smoothness and renormings in
Banach spaces}, Pitman Monographs and Surveys in Pure and Applied Mathematics, vol. 64, Longman Scientific \& Technical, Harlow 1993.

\bibitem{DK} S. Draga, T. Kochanek, \emph{The Szlenk power type and tensor products of Banach spaces}, Proc. Amer. Math. Soc.~{\bf 145} (2017), 1685--1698.

\bibitem{Enflo} P. Enflo, \emph{Banach spaces which can be given an equivalent uniformly convex norm}, Israel J. Math. {\bf 13} (1972), 281--288. 

\bibitem{ELP} P. Enflo, J. Lindenstrauss, G. Pisier, \emph{On the ``three space problem"}, Math. Scand. \textbf{36} (1975), 199--210. 

\bibitem{FJ} T. Figiel, W.B. Johnson, \emph{A~uniformly convex Banach space which contains no $\ell_p$}, Compositio Math.~{\bf 29} (1974), 179--190.

\bibitem{GKL}  G. Godefroy, N.J. Kalton, G. Lancien, \emph{Szlenk indices and uniform homeomorphisms}, Trans. Amer. Math. Soc. \textbf{353} (2001), 3895--3918.

\bibitem{hsmvz} P. H\'ajek, V. Montesinos Santaluc\'ia, J. Vanderwerff, V. Zizler, \textit{Biorthogonal Systems in Banach Spaces}, CMS Books in Mathematics, vol.~26, Springer, New York 2008.

\bibitem{Hinrichs} A. Hinrichs, \emph{Operators of Rademacher and Gaussian subcotype}, J. Lond. Math. Soc., \textbf{63} (2001), 453-468. 

\bibitem{James} R. C. James. \emph{Uniformly non-square Banach spaces}, Ann. of Math.  \textbf{80} (1964), 542-550.

\bibitem{KOS} H. Knaust, Th. Schlumprecht, E. Odell, \emph{On Asymptotic structure, the Szlenk index and UKK properties in Banach spaces}, Positivity \textbf{3} (1999), 173--199.

\bibitem{lancien} G. Lancien, \emph{A~survey on the Szlenk index and some of its applications}, Rev. R.~Acad. Cien. Serie~A. Mat. {\bf 100} (2006), 209--235.

\bibitem{MMT-J} B. Maurey, V.D. Milman, N.~Tomczak-Jaegermann, \emph{Asymptotic infinite-dimensional theory of Banach spaces}, Oper. Theory Adv. Appl.~{\bf 77} (1994), 149--175.

\bibitem{MP} B. Maurey, G. Pisier, \emph{S\'eries de variables al\'eatoires vectorielles ind\'ependantes et propri\'et\'es g\'eom\'etriques des espaces de Banach}, Studia Math.~{\bf 58} (1976), 45--90.

\bibitem{MTJ} V.D. Milman, N.~Tomczak-Jaegermann, \emph{Asymptotic $\ell_p$ spaces and bounded distortions}, Banach Spaces (M\'erida, 1992; Bor-Luh Lin and W.B.~Johnson, eds.), Contemp. Math.~{\bf 144} (1993), 173--195.

\bibitem{MS} V.D. Milman, G. Schechtman, \emph{Asymptotic theory of finite dimensional normed spaces}, Lecture Notes in Mathematics~1200, Springer-Verlag, Berlin Heidelberg~1986.

\bibitem{OS_trees} E. Odell, Th. Schlumprecht, \emph{Trees and branches in Banach spaces}, Trans. Amer. Math. Soc.~{\bf 354} (2002), 4085--4108.

\bibitem{OS_handbook} E. Odell, Th. Schlumprecht, \emph{Distortion and asymptotic structure}, in: Handbook of the geometry of Banach spaces, vol.~II (edited by W.B.~Johnson and J.~Lindenstrauss), Ch.~31, pp.~1333--1360, North Holland~2003.

\bibitem{pisier} G. Pisier, \emph{Martingales with values in uniformly convex spaces}, Israel J.~Math.~{\bf 20} (1975), 326--350.

\bibitem{PX} G. Pisier, Q. Xu, \emph{Random series in the real interpolation spaces between the spaces $v_p$}, in: Geometrical Aspects of Functional Analysis. Israel Seminar, 1985-86, J. Lindenstrauss and V.D. Milman (eds.), Lecture Notes Math. 1267, pp. 185--209, Springer-Verlag, Berlin 1987. 

\bibitem{R} M. Raja, \emph{On weak$^\ast$ uniformly Kadec-Klee renormings}, Bull. London Math. Soc. \textbf{42} (2010), 221--228.

\bibitem{Sz} W. Szlenk, \emph{The non-existence of a separable reflexive Banach space universal for all separable reflexive Banach spaces}, Studia Math. \textbf{30} (1968), 53--61.

\bibitem{Wenzel} J. Wenzel, \emph{Uniformly convex operators and martingale type}, Rev. Mat. Iberoam. \textbf{18} (2002), 211--230. 
\end{thebibliography}
\end{document}